\documentclass{amsart}
\usepackage{amsmath}
\usepackage{amssymb}
\usepackage{amsthm}

\usepackage{mathrsfs}
\usepackage{enumerate}
\usepackage[labelfont=bf]{caption}
\usepackage{accents}
\usepackage{microtype}
\usepackage{mathtools}

\usepackage{hyperref}
\usepackage{amsmath}
\usepackage{amsthm}
\usepackage{amssymb}
\usepackage{xfrac}
\usepackage{mathtools}
\usepackage{tikz-cd}
\usepackage{enumitem}
\usepackage{comment}
\usepackage{scalerel}
\usepackage{stackengine,wasysym}
\usepackage{csquotes}

\usepackage{todonotes}

\newcommand{\real}{\mathbb{R}}
\newcommand{\ints}{\mathbb{Z}}

\DeclareMathOperator{\GL}{GL}
\DeclareMathOperator{\PSL}{PSL}
\DeclareMathOperator{\SL}{SL}

\DeclareMathOperator{\U}{U}
\DeclareMathOperator{\Real}{Re}
\DeclareMathOperator{\sub}{Sub}

\DeclareMathOperator{\cl}{Cl}

\DeclareMathOperator{\Stab}{Stab}

\DeclareMathOperator{\op}{op}

\DeclareMathOperator{\supp}{supp}
\DeclareMathOperator{\Prob}{Prob}

\DeclareMathOperator{\IRS}{IRS}

\newcommand{\Sub}[1]{\mathrm{Sub}(#1)}

\newcommand{\Cl}[1]{\mathrm{Cl}(#1)}

\theoremstyle{plain}
\newtheorem{thm}{Theorem}[section]
\newtheorem{theorem}[thm]{Theorem}
\newtheorem{lem}[thm]{Lemma}
\newtheorem{lemma}[thm]{Lemma}
\newtheorem*{lemma*}{Lemma}
\newtheorem{remark}[thm]{Remark}
\newtheorem{prop}[thm]{Proposition}
\newtheorem{cor}[thm]{Corollary}
\newtheorem{example}[thm]{Example}
\newtheorem*{thm*}{Theorem}

\numberwithin{equation}{section}
\theoremstyle{definition}
\newtheorem{defn}[thm]{Definition}
\newtheorem{notation}[thm]{Notation}

\newtheorem{eg}[thm]{Example}
\newtheorem{rmk}[thm]{Remark}

\title{Spectral gap for products and a strong normal subgroup theorem}
\author{Uri Bader, Tsachik Gelander and Arie Levit}

\begin{document}

\begin{abstract}
We establish a general spectral gap theorem for actions of products of groups which may replace Kazhdan's property (T) in various situations. As a main application, we prove that a confined subgroup of an irreducible lattice in a higher rank semisimple Lie group is of finite index. This significantly strengthens the classical normal subgroup theorem of Margulis and removes the property (T) assumption from the recent counterpart result of Fraczyk and Gelander. We further show that any confined discrete subgroup of a higher rank semisimple Lie group satisfying a certain irreducibility condition is an irreducible lattice. This implies a variant of the Stuck--Zimmer conjecture under a strong irreducibility assumption of the action.
\end{abstract}

\maketitle

\begin{center}\textbf{Dedicated to Gregory Margulis with great admiration and affection.}\end{center}

\section{Introduction}
\label{sec:intro}

A subgroup $\Lambda$ of a given discrete group $\Gamma$ is called \emph{confined} if there exists a finite subset $F \subset \Gamma \setminus \{e\}$ such that $\Lambda^\gamma \cap F \neq \emptyset$ holds true for every element $\gamma \in \Gamma$. One of our main results is the following:

\begin{theorem}
\label{thm intro: lattices}
Let $G$ be a connected semisimple  Lie group of real rank at least two  and with trivial center. Let $\Gamma$ be an irreducible lattice in $G$. Then any confined subgroup of  $ \Gamma$ has finite index.
\end{theorem}
 
Note that every non-trivial normal subgroup is confined. So Theorem \ref{thm intro: lattices} vastly extends the celebrated normal subgroup theorem   of Margulis \cite{margulis1978quotient,margulis1979finiteness}, which says in turn that every non-trivial normal subgroup of a center-free higher-rank irreducible lattice has finite index. 
Loosely speaking, this theorem and  its generalizations \cite{SZ,burger2000lattices,bekka2007operator,creutz2017stabilizers,fraczyk2023infinite} are consequences of the conflict between two incompatible analytic properties, namely \enquote{amenability} and \enquote{spectral gap}. Hence these results are  easier to obtain if the semisimple group $G$, or least one of its factors, has Kazhdan's property (T).  The general case of the normal subgroup theorem where the group $G$ does not have property (T) required    special  attention, which Margulis carried out in \cite{margulis1979finiteness}. In the classical context of lattices in higher rank semisimple Lie groups, this has been the state of the art until today;  no improvement  has been made in the absence of property (T). \footnote{The works \cite{burger2000lattices,shalom2000rigidity,bader2006factor}  do not rely on Kazhdan's property (T). However,  they improve the  normal subgroup theorem in a different direction than the current paper, i.e. they extend the class of groups to which the theorem applies, while here we obtain a  stronger result in  the classical case of semisimple Lie groups.}

Here is another way to think about  Theorem \ref{thm intro: lattices}. Given a discrete group $\Gamma$  we consider the  space of its subgroups $\Sub{\Gamma}$, called its \emph{Chabauty space}. Put a compact topology on $\Sub{\Gamma}$ by identifying it with a closed subset of the Cantor space $\{0,1\}^\Gamma$. The group $\Gamma$ acts on its Chabauty space  by homeomorphisms via conjugation. In this language,  a subgroup $\Lambda \in \Sub{\Gamma}$ is \emph{unconfined} (i.e. not confined) if and only if the trivial subgroup of $\Gamma$ belongs to the closure of the $\Gamma$-orbit of the subgroup $\Lambda$ in the Chabauty space. 

\addtocounter{thm}{-1}
\begin{theorem}[Reformulation]
Let $\Gamma$ be an irreducible lattice in a connected, center-free semisimple Lie group of real rank at least two.
Then any infinite index subgroup of  $\Gamma$ is unconfined, i.e. admits a sequence of conjugates converging to the trivial subgroup in the Chabauty topology.
\end{theorem}

A \emph{uniformly recurrent subgroup} (URS) of $\Gamma$ is a closed minimal  $\Gamma$-subsystem of $\Sub{\Gamma}$ \cite{glasner2015uniformly}. Note that if $X \subset \Sub{\Gamma}$ is a non-trivial uniformly recurrent subgroup then any subgroup $\Lambda \in X$ is confined.

\begin{cor}
\label{cor:lattices URS statement }
Let $\Gamma$ be an irreducible lattice as in Theorem \ref{thm intro: lattices}. Any non-trivial uniformly recurrent subgroup $X \subset \Sub{\Gamma}$ is of the form $X = \{\Lambda^\gamma \: : \: \gamma \in \Gamma\} $ where $\Lambda $ is some finite-index subgroup of $\Gamma$.
\end{cor}

The above results hold true more generally for $S$-arithmetic subgroups of semisimple algebraic groups over local fields of zero characteristic. For a rigorous statement of the results in that generality see Theorem \ref{thm: lattices, general} below.

In the special case where $G$ has a simple factor with Kazhdan's property (T), these results follow from the recent work of Fraczyk--Gelander \cite{fraczyk2023infinite} as well as the work of Bader--Boutonnet--Houdayer--Peterson \cite{BBHP}. 
The breakthrough of the current paper is that it applies to all higher rank groups regardless of property (T). 

The preceding discussion concerns subgroups which are a priori contained in a given lattice. We now turn to consider more general discrete subgroups. A suitable adaptation of the notion of confined subgroups is required here. A subgroup $\Lambda$ of a second countable locally compact group $G$ is called \emph{confined} if the trivial subgroup   is not in the Chabauty closure of the orbit $\Lambda^G$ via conjugation.
When $G$ is a semisimple Lie group, a discrete subgroup   $\Lambda \le G$ is confined if and only if the corresponding locally symmetric space $\Lambda\backslash G/K$ has a uniform upper bound on its injectivity radius at all points (where $K$ is a maximal compact subgroup of $G$).
It is proven in \cite{fraczyk2023infinite} that confined discrete subgroups of simple center-free Lie groups of real rank at least two are lattices; see \cite{fraczyk2023infinite} for a more refined statements and for analogs in the semisimple setup where property (T) is assumed.

To state our next result
we will need a sharpening of the confined condition. It is   used to ensure that a discrete subgroup of a product does not degenerate into any proper factor when taking conjugates. The term \emph{conjugate limit} in Definition \ref{def:strongly and irreducibly confined - intro} stands for any subgroup  in the Chabauty orbit closure under conjugation of the given subgroup. 
 
\begin{defn}
\label{def:strongly and irreducibly confined - intro}
A subgroup $\Lambda$ of a  locally compact second countable group $G$ is  \emph{strongly confined} if no conjugate limit of $\Lambda$ is contained in a proper normal subgroup of $G$. It is  \emph{irreducibly confined}
    if it is strongly confined and
    the  intersection $\Lambda \cap H$ is trivial for   any   proper normal subgroup $H\lhd G$.
\end{defn}

Certainly, every strongly confined subgroup is confined, and every confined subgroup of a simple Lie group is strongly (and irreducibly) confined. 

 \begin{theorem}
\label{thm intro:generalLie}
Let $G$ be a connected semisimple  Lie group of real rank at least two  and with trivial center. A discrete subgroup $\Lambda$ of $G$   is irreducibly confined if and only if $\Lambda$ is an irreducible lattice. 
\end{theorem}

We refer to Theorem \ref{thm:general with non-zarsiki dense intersection} below for a   stronger but somewhat more involved version of Theorem \ref{thm intro:generalLie} (relying  on a weaker assumption than irreducibly confined).



Moving beyond Lie groups we have fewer techniques at our disposal. However, it is possible to get around this by requiring a stronger irreducibility condition.
 

\begin{theorem} \label{thm:strong irreducible}
Let $G = G_1 \times G_2$ be a product of two second countable locally compact  groups. Let $\Lambda \le G$ be a discrete coamenable  subgroup. Assume that $G_2$ has a compact abelianization and that there are no $G_2$-invariant vectors in $L^2_0(G/\Lambda)$. If  every  conjugate limit of $\Lambda$ projects densely to the factor $G_2$ then $\Lambda$ is a lattice in $G$.
\end{theorem}


The novelty of the current work is to rely on the product structure of the group $G$ as a replacement for property (T). The main breakthrough  is the following general spectral gap result for  actions of product groups, under certain assumptions on the stabilizer structure of the action.


\begin{theorem}[Spectral gap for actions of products]
\label{theorem:getting spectral gap}
Let $G_1$ and $G_2$ be a pair of locally compact second countable compactly generated groups such that $G_2$ has compact abelianization. Set $G = G_1 \times G_2$. 
Let $X$ be a locally compact topological $G$-space endowed with a $G$-invariant (finite or infinite) measure $m$. Assume that
\begin{itemize}
    \item $L^2_0(X,m)^{G_2}=0$, and
    \item there is a closed $G$-invariant subset of $\Sub G$ containing $\mathrm{Stab}_G(x)$ for \mbox{$m$-almost} every  point $x\in X$ such that every subgroup $H$ in  this subset satisfies $\overline{G_1 H} = G$.
\end{itemize}
Then the unitary Koopman $G$-representation $L^2_0(X,m)$ has a spectral gap.
\end{theorem}

A more technically demanding but sharper statement is Theorem \ref{theorem:getting spectral gap - general case} below. Additionally, in the context of products of semisimple real or $p$-adic Lie groups, we obtain the stronger Theorem \ref{theorem:getting spectral gap - analytic groups}.




Recall that a probability measure preserving  action of a center-free semisimple Lie group $G$ is called irreducible if every simple factor of $G$ acts ergodically. The rigidity theorem of Stuck and Zimmer \cite{SZ} says that if $G$ has higher rank  and  property (T) then every irreducible probability measure preserving action of $G$ is either essentially free or essentially transitive. 
Hartman and Tamuz \cite{Hartman-Tamuz} showed that it is enough to suppose that one of the simple factors of $G$ has property (T).
The famous Stuck--Zimmer conjecture says that the rigidity theorem should apply to all higher rank semisimple Lie groups regardless of property (T).

Let us say that a discrete subgroup of a semisimple Lie group $G$ is {\it irreducible} if it projects densely to every proper factor of $G$.
An ergodic action of $G$ with almost surely irreducible stabilizers is irreducible  \cite[\S7]{fraczyk2023infinite}. 
Let us say that a discrete subgroup $\Lambda\le G$ is {\it strongly irreducible} if every discrete conjugate limit of $\Lambda$ is irreducible, and that a probability measure preserving action is strongly irreducible if almost every stabilizer is strongly irreducible.

\begin{cor}[A weak version of the Stuck--Zimmer conjecture]
\label{cor:intro: weak version}
Let $G$ be a  connected center-free semisimple Lie group of real rank at least two. Then every   strongly irreducible probability measure preserving action of $G$ is essentially transitive. 
\end{cor}

See Theorem \ref{theorem:strictly confined coamenable is a lattice} below for an alternative weak version of the Stuck--Zimmer conjecture. It says that an irreducible invariant random subgroup of a connected center-free higher rank semisimple Lie group is supported on lattices if and only if it is almost surely irreducibly confined (in the sense of Definition \ref{def:strongly and irreducibly confined - intro}).

\subsection*{Structure of the paper}

In \S\ref{sec:prelim} we address various preliminaries such as   unitary representations, spectral gap and  asymptotically invariant vectors as well as the Chabauty topology. In \S\ref{sec:Q} we make some remarks regarding the representation theory of direct product of groups with compact abeliazniations. The technical heart of this paper is \S\ref{sec:Koopman reps} in which we prove Theorem~\ref{theorem:getting spectral gap}, as well as its generalization, Theorem~\ref{theorem:getting spectral gap - general case}.
In \S\ref{sec:standard semisimple group} we introduce our working notion of standard semisimple groups. In \S\ref{sec:geometry of discrete subgroups of products} we study geometric properties of discrete subgroups of standard semisimple  groups and establish our spectral gap theorem for actions of products of semisimple groups. In
\S\ref{sec:confined and irreducibly} we introduce confined and strongly confined subgroups. In \S\ref{sec:confined subgroups of lattices} and \S\ref{sec:strongly confined} we deal with confined subgroups of lattices and strongly confined subgroups of semisimple Lie groups, respectively.
Theorem \ref{thm intro: lattices} and   Corollary \ref{cor:lattices URS statement }
are proven in \S\ref{sec:confined subgroups of lattices}
and Theorem \ref{thm intro:generalLie}
and Theorem \ref{thm:strong irreducible}
are proven in \S\ref{sec:strongly confined}.
Lastly, the standalone \S\ref{sec:margulis functions} takes up the notion of Margulis functions needed to ensure discreteness in a certain argument in \S\ref{sec:strongly confined}.



\section{Preliminaries}
\label{sec:prelim}

We set up some basic notions and terminology to be used throughout the entire paper.

\subsection*{The Chabauty space}

Let $G$  be a locally compact second countable group. 
The Haar measure on $G$ will be denoted by $m_G$. 
Typically, the group $G$ will have a compact abelianization, and in particular it will be unimodular.

A discrete subgroup of the group $G$ will typically be denoted by $\Lambda$. The Haar measure on the quotient $G/\Lambda$  will be denoted by  $m_{G/\Lambda}$.
If the measure $m_{G/\Lambda}$ is  finite we say that $\Lambda$ is a \emph{lattice} in $G$.
A lattice will  typically be denoted by $\Gamma$.
In case the quotient $G/\Gamma$ is compact the lattice $\Gamma$ is called \emph{uniform}.

We  denote by $\Sub G$ the space of all closed subgroups of the group $G$ endowed with the Chabauty topology \cite{chabauty1950limite}. Recall that the space $\Sub G$ is compact, and the group $G$ acts on it by homeomorphisms via conjugation. A non-empty minimal closed $G$-invariant subset of $\Sub G $ is called an \emph{uniformly recurrent subgroup (URS)} of $G$, see  \cite{glasner2015uniformly}. We will use the notation 
\[\Lambda^G = \{\Lambda^g \: : \: g \in G\} \subset \Sub{G}\]
for the $G$-orbit under conjugation of a given subgroup $\Lambda \in \Sub{G}$.
\begin{defn}
\label{def:conjugate limit}
A \emph{conjugate limit} of a subgroup $\Lambda \le G$ is any subgroup $\Delta \in \overline{\Lambda^G}$.
\end{defn}

We denote by $\Prob(\Sub G)$ the space of all probability measures on $\Sub{G}$ endowed with the weak-$*$ topology. This makes $\Prob(\Sub G)$  a compact convex space. It is regarded  with the natural $G$-action.
A $G$-fixed point in this space is called an \emph{invariant random subgroup (IRS)} of $G$. We denote the space $\Prob(\Sub G)^G$ of all invariant random subgroups  by $\IRS(G)$. See \cite{abert2014kesten,7s,7s-b,KM-IRS,gelander2018invariant} or the surveys \cite{gelander2018view,gelander-lecture-IRS}.

Say that $\Gamma$ is a lattice in $G$ and $\nu \in \IRS(\Gamma)$ is an invariant random subgroup of $\Gamma$. It is possible to induce  $\nu$ and obtain an \emph{induced invariant random subgroup} $\overline{\nu} \in \IRS(G)$. This is done as follows. View $\Sub{\Gamma}$ as a subset of $\Sub{G}$ and regard $\nu$ as a probability measure on the space $\Sub{G}$. Fix an arbitrary Borel fundamental domain $\mathcal{F} \subset G$ for the lattice $\Gamma$ and normalize the Haar measure so that $m_G(\mathcal{F})=1$. Finally take $\overline{\nu} = \int_{\mathcal{F}} g_* \nu \;\mathrm{d} m_G(g)$.

Recall that a probability measure preserving Borel $G$-space is called \emph{irreducible} if every non-trivial normal subgroup of $G$ is acting ergodically. An invariant random subgroup $\nu$ is \emph{irreducible} if the Borel $G$-space $(\Sub{G},\nu)$ is irreducible.

Let $\mu$ be probability measure on the group $G$. We will typically assume that $\mu$ is absolutely continuous with respect to the Haar measure and that its support generates the group $G$. A \emph{$\mu$-stationary random subgroup} of the group $G$ is a probability measure $\nu \in \Prob(\Sub G)$ satisfying $\mu * \nu = \nu$. For some recent works dealing with stationary random subgroups see \cite{gelander2022effective,fraczyk2023infinite,gekhtman2023stationary}.

\subsection*{Algebras of  functions}

We denote by $C_c(G)$   the $\mathbb{C}$-algebra of all  compactly supported continuous complex-valued functions on the group $G$  with the algebra product given by convolution.
We regard $C_c(G)$ as a normed algebra with respect to the topology of uniform convergence on compact subsets. We also endow $C_c(G)$ with the supremum norm $\|\cdot \|_\infty$.
For a function $f$ on the group $G$ we write $\check{f}=f\circ \iota$, where $\iota:G\to G $ is the inversion map $\iota : g \mapsto g^{-1}$. We set
\[ \mathcal{A}(G) = \{ f \in C_c(G) \: : \:  f\ge 0, \; f=\check{f} \;\, \text{and} \;\, m_G(f) = 1\}. \]

A function $f\in C_c(G)$ is   \emph{generating} if its support $\mathrm{supp}(f)$ generates the group $G$.
Note that if the group $G$ is connected then any  non-zero function is generating. 


\subsection*{Unitary representations}
\label{subsection:unitary reps}

Vector spaces are taken over the complex numbers. 
In particular, the Banach algebra $L^1(G)$ and the Hilbert spaces $L^2(G/\Lambda)$ where $\Lambda$ is some discrete subgroup of $G$ are taken with complex  coefficients.
It is tacitly assumed that these spaces are taken with respect to the Haar measures $m_G$ and $m_{G/\Lambda}$ correspondingly. The corresponding norms are denoted by $\|\cdot\|_1$ and $\|\cdot \|_2$.

Hilbert spaces will  typically be denoted by $V$ and  assumed to be separable. 
An unindexed norm $\|\cdot\|$ is typically associated with a Hilbert space which should be clear from the context.
We will denote by $\mathrm{B}(V)$ the algebra of bounded operators on the Hilbert space $V$ and by $\|\cdot\|_{\op}$ the operator norm on $\mathrm{B}(V)$. 
We denote by $\U(V)$ the group of unitary operators in $\mathrm{B}(V)$
 endowed with the strong operator topology.
This is a Polish topological group.

By a \emph{unitary representation} we mean a continuous homomorphism $G\to \U(V)$.
By an obvious abuse of notation, given such a unitary representation, an element $g\in G$ and a vector $v\in V$, we denote by $gv$ the image of $v$ under the unitary operator associated with $g$.  
Such a unitary representation  extends to a representation  of the algebra of complex-valued measures of bounded total  variation on the group $G$.
In particular, probability measures on the group $G$ act on the Hilbert space $V$ via averaging operators. Symmetric probability measures give rise to self-adjoint operators of norm at most one.

Regarding elements of the Banach algebra $L^1(G)$ as densities of $m_G$-absolutely continuous measures on the group $G$, we get a  representation $L^1(G)\to \mathrm{B}(V)$.  
In particular, an element $f \in \mathcal{A}(G)$ gives rise to a self-adjoint operator of norm at most  $1$, which we regard as a smooth averaging operator.
By an abuse of notation, given any vector $v\in V$ we denote by $f v$ the image of $v$ under this operator.  

\begin{lemma} \label{lem:mazur}
    Let $X$ be a space endowed with a positive measure (either finite or infinite).
    The map
    \[ \{f\in L^2(X)\: : \: \|f\|_2=1,~f\geq 0\} \to \{f\in L^1(X) \: : \: \|f\|_1=1,~f\geq 0\}, \quad f \mapsto f^2\]
    is a uniform homeomorphism, namely this map and its  inverse are uniformly continuous.
\end{lemma}

The domain and the range of the map $f \mapsto f^2$  are  endowed with the $\|\cdot\|_2$-metric and the  $\|\cdot\|_1$-metric, respectively.

\begin{proof}[Proof of Lemma \ref{lem:mazur}]
Consider any pair of functions $f,g \in L^2(X)$ with $\|f\|_2=\|g\|_2 = 1$ and $f, g \ge 0$. Observe that the pair $f,g$  satisfies
  \[ \|f^2-g^2\|_1=\langle |f+g|,|f-g| \rangle \leq \|f+g\|_2\cdot \|f-g\|_2 \leq 2 \|f-g\|_2.   \]
Therefore the map  $f \mapsto f^2$ is uniformly continuous. 
  To show that its inverse  is uniformly continuous, we use the inequality
  $ |a-b|^2\leq |a^2-b^2| $  which is valid for any pair of real numbers $a,b\geq 0$. This implies that the pair of functions $f,g$ satisfies
  \[ \|f-g\|_2 \leq \|f^2-g^2\|_1^{1/2}. \]
  This concludes the proof.
\end{proof}

\subsection*{Asymptotically invariant vectors}
Let  $G\to \U(V)$ be a unitary representation of the group $G$ on the Hilbert space $V$.

\begin{defn}
A sequence of non-zero vectors $v_n \in V$ is called
\emph{asymptotically $G$-invariant} if 
for every compact subset $K\subset G$ 
\[ \lim_n \sup_{k \in K} \|(1-k)v_n\| / \|v_n\| = 0.\]
\end{defn}

Fix a generating function $\phi\in \mathcal{A}(G)$. It is well known that a sequence of non-zero vectors $v_n \in V$ is asymptotically $G$-invariant if and only if
\[ \lim_n \|(1-\phi) v_n\| / \|v_n\|  = 0. \]

In case such an asymptotically $G$-invariant sequence exists, we say that the representation $V$ \emph{almost has $G$-invariant vectors}.
Otherwise, we say that it has a \emph{spectral gap}. Indeed, spectral gap is equivalent to saying that $\|\phi\|_{\op}<1$ or to the fact that $0$ is not contained in the spectrum of the positive operator $1-\phi$.

\begin{defn}[{\cite[Definition IV.3.5]{margulis1991discrete}}]
A subset $A\subset V \setminus \{0\}$ is said to be \emph{$G$-uniform} if for every $\varepsilon>0$ there exists an identity neighborhood $U\subset G$ such that 
\[ \sup_{v \in A}  \sup_{g \in U}  \|(1-g)v\| / \|v\| \leq \varepsilon. \]
A sequence of vectors $v_n \in V$ is said to be \emph{$G$-uniform} if the set $\{v_n\}$ is.
\end{defn}
Another way to think about this definition is to say that orbit maps of  vectors from the subset $A$ mapping into the Hilbert space $V$ are uniformly equicontinuous.
It is easy to see that every asymptotically \mbox{$G$-invariant} sequence in $V$ is $G$-uniform.

Here are some elementary lemmas concerning the above notions.  


 \begin{lemma}
 \label{lemma:adding a zero sequence}
 Let $v_n, u_n \in V$ be two sequences of vectors with $\liminf_n \|v_n\| > 0$ and $\limsup_n \|u_n\| = 0$. If the sequence $v_n$ is asymptotically $G$-invariant (respectively $G$-uniform) then the sequence $v_n + u_n$ has the same property.
 \end{lemma}

\begin{proof}
Fix a generating function $\varphi \in \mathcal{A}(G)$. Assume to begin with that the sequence $v_n$ is asymptotically $G$-invariant. For all $n$ sufficiently large so that $\|u_n\| \le \frac{1}{2} \|v_n\|$ we have $\|v_n + u_n\| \ge \frac{1}{2}\|v_n\|$ and
\begin{align}
\begin{split}
\frac{\|(1-\varphi) (v_n + u_n) \|}{ \|v_n + u_n\| } &\le 2 \frac{\|\varphi v_n - v_n\| + \| \varphi u_n\| + \|u_n\|}{\|v_n\|} \le \\
&\le 2 \frac{\|(1-\varphi) v_n \|   + 2\|u_n\|}{\|v_n\|}.
\end{split}
\end{align}
We have used the fact that $\varphi$ gives rise to a contracting operator so that $\|\varphi\|_\textrm{op} \le 1$. Letting $n \to \infty$  shows that the sequence $v_n + u_n$ is indeed asymptotically $G$-invariant.


Next, assume that the sequence $v_n$ is $G$-uniform. 
The verification of the fact that the sequence $v_n + u_n$ is also $G$-uniform is very similar to the above computation, up to considering the operator $g$ for some sufficiently small element $g \in G$ instead of the averaging operator $\varphi$.
\end{proof}

 \begin{lemma}
 \label{lemma:not invariant means a particular element}
 Let $v_n \in V$ be a $G$-uniform sequence of non-zero vectors. If the sequence $v_n$ is \emph{not} asymptotically $G$-invariant then there is some element $g \in G$ such that 
 \[\limsup_n \|(1-g) v_n\| / \|v_n\| > 0.\]
 \end{lemma}
\begin{proof}
Assume that the sequence $v_n$ is not asymptotically $G$-invariant. This means that there is some compact subset $ K \subset G$ and some $\varepsilon > 0$ such that
\[ \limsup_n \sup_{ g \in K} \|(1-g) v_n  \| / \|v_n\| > \varepsilon. \]
Since the sequence $v_n$ is $G$-uniform we may find a symmetric identity neighborhood $U \subset G$ such that 
\[\sup_n \sup_{g \in U}  \|(1-g)v_n\| / \|v_n\| \leq \varepsilon/2. \]
Let $U g_1 ,\ldots,U g_N $ be a finite cover of the compact subset $K$ for some choice of elements $g_1,\ldots,g_N \in G$. It follows that one of these elements $g_i$ is as required.
\end{proof}

It is useful to note that both properties of being asymptotically $G$-invariant as well as that of being $G$-uniform are preserved under rescaling (i.e. $v_n \mapsto c_n v_n$ for some arbitrary scaling constants $c_n > 0$).


\begin{lemma}
\label{lemma:length function}
Let $D \subset G$ be any subset. For each vector $v \in V$ and for all $n \in \mathbb{N}$ we have
\begin{equation*}  
\sup_{g \in D^n} \| (1-g)v \| \le n \sup_{g \in D} \| (1-g)v \|.
\end{equation*}\end{lemma}
\begin{proof}
Fix a vector $v \in V$. By the triangle inequality,  any pair of elements $g,h \in G$ satisfies
\begin{equation*}  
   \|(1-gh)v\|\leq \|(1-g)v\|+\|g(1-h)v\| = \|(1-g)v\|+\|(1-h)v\|.
\end{equation*}
The desired conclusion follows by induction on $n$.
\end{proof}

We will require the following lemma dealing with  asymptotic invariance in the $L^1$-sense.

\begin{lemma} \label{lem:L1SG}
Let $G$ be a locally compact group admitting a  measure preserving action on a probability measure space $(X,m)$.
Assume that the   representation $L^2_0(X)$ has  spectral gap. If $v_i \in L^1(X)$ is an asymptotically $G$-invariant sequence of vectors with $\|v_i\|_1 = 1$ and $v_i \ge 0$ then it converges in $L^1(X)$ to the constant function $1$.
\end{lemma}

In the above statement, the notion of an asymptotically $G$-invariant sequence is understood in the $L^1$-sense.

\begin{proof}
The lemma is a direct consequence of Lemma~\ref{lem:mazur}. Indeed, the sequence $v_i^{\frac{1}{2}}\in L^2(X)$ is asymptotically invariant (in the $L^2$-sense). Therefore $\|v_i^{\frac{1}{2}} - 1\|_2 \to 0$. Another application of Lemma~\ref{lem:mazur} gives $\|v_i-1\|_1\to 0$, as required.
\end{proof}

\section{Unitary representations of product groups}
\label{sec:Q}

In this section we establish the following special property of the  representation theory of product groups.

\begin{lemma}
\label{lem:mar3.7}
    Let $G=G_1\times G_2$ where $G_1$ and $G_2$ are compactly generated locally compact groups. Let $G\to \U(V)$ be a unitary representation without spectral gap and with $V^{G_2}=0$.
    If $G_2$ has  compact abelianization then there exists a sequence of unit vectors in $V$ which is $G$-uniform, asymptotically $G_1$-invariant and \emph{not} asymptotically $G_2$-invariant.
\end{lemma}

This  lemma is a version of \cite[Lemma IV.3.7]{margulis1991discrete}.
Margulis proves this result  for groups which form a Gelfand pair with respect to a compact subgroup. He refers to this condition as 
   \emph{property (Q)}. Our version is more general, as we replace this assumption by  compact abelianization. 
We will call a sequence of vectors with the peculiar properties provided by Lemma \ref{lem:mar3.7} a \emph{discordant sequence}. The lemma  is proved at the end of this section.  

\begin{remark}
\label{remark:non trivial part 1}
The assumption that the representation $G \to \mathrm{U}(V)$ has no spectral gap implicitly implies that the Hilbert space $V$ is non-zero. Therefore the additional assumption  $V^{G_2} = 0$  forces the group $G_2$ to be non-trivial. We allow $G_1$ to be trivial.
\end{remark}


\subsection*{Compact abelianization and unitary representations}

Let $G$ a compactly generated locally compact group with compact abelianization. Let $K \subset G$ be a compact, symmetric and generating    identity neighborhood. 
The following lemma shows that the compact abelianization property can  be tracked on a compact subset.

\begin{lemma} \label{lem:ArAs}
There is a symmetric compact subset $Q \subset G$ and a constant $\delta > 0$ with the following property --- for every continuous function $\phi: Q \to \mathbb{C}$ satisfying $\|\phi|_K\|_\infty = 1$ there is  a pair of elements  $k_1 \in K$ and $k_2 \in Q$  with
    \[ |\phi(k_1)+\phi(k_2)-\phi(k_1k_2)|\geq \delta. \]
\end{lemma}

\begin{proof}
We set inductively $K_1=K$ and $K_n=K_1\cdot K_{n-1}$ for all $n \in \mathbb{N}$.
Assume by contradiction that there exists a sequence of continuous functions $\phi_n:K_n\to \mathbb{C}$ satisfying $ \|\phi_n|_K\|_\infty=1$ and such that
\[ \sup_{(k_1,k_2) \in K \times K_{n-1}}|\phi_n(k_1)+\phi_n(k_2)-\phi_n(k_1k_2)| \le  1/n \]
 for every $n \in \mathbb{N}$. 
Fix an index $i \in \mathbb{N}$. We claim that: 
\begin{itemize}
    \item The sequence $(\phi_n|_{K_i})_{n >  i}$ is uniformly bounded on $K_i$.  Indeed for each $n$, using $\|\phi_n|_K\|_\infty = 1$ and the triangle inequality  we see that 
    \[ \|\phi_n|_{K_i}\|_\infty \leq i+(i-1)/n < 2i. \]
    \item The  sequence $(\phi_n|_{K_i})_{n > i}$ is equicontinuous on $K_{i-1}$. For each $m \in \mathbb{N}$ let $U_m \subset G$ be a sufficiently small identity neighborhood so that $U_m^m \subset K$. Provided that $n \ge m$, the triangle inequality implies that every element $h \in U_m$ satisfies
\[ |m\phi_n(h)-\phi_n(h^m)| \leq (m-1)/n <1.\]   
Thus using $\|\phi_n|_K\|_\infty= 1$ we get
\[ |m\phi_n(h)|\leq |m\phi_n(h)-\phi_n(h^m)|+|\phi_n(h^m)|< 2 \]
 for every element $h \in U_m$. We conclude that $\|\phi_n|_{U_m}\|_\infty\leq 2/m$ for all $n\geq m$. 
  Therefore   every pair of elements $g \in K_i$ and $h \in U_m$ satisfies 
\[ |\phi_n(g)-\phi_n(hg)| < 1/n + |\phi_n(h)| \le 1/n + 2/m \le 3/m \]
provided that $n \ge m$. The desired equicontinuity follows.
\end{itemize}

We use the Arzela--Ascoli theorem to conclude that for each fixed $i \in \mathbb{N}$, the sequence $(\phi_n|_{K_i})_{n > i}$ has a uniformly convergent subsequence in $C(K_i)$.
Hence, by a standard diagonal argument, the sequence $\phi_n$ admits a convergent subsequence with respect to the topology of uniform convergence on compact subsets of $G$. The triangle inequality shows that the limit of this subsequence is a continuous homomorphism from $G$ to the additive group of $\mathbb{C}$. This homomorphism is non-trivial as 
 $\|\phi_n|_K\|_\infty=1$ for all $n$. This is a contradiction.
 
To finish the proof, we may take $ Q = K_n$ and $\delta = \frac{1}{n}$ for some suitable index $n$ where our assumption toward contradiction fails.
\end{proof}

The following proposition shows that in any unitary representation of a group with compact abelianization,
every non-invariant vector can  be transformed in a uniform fashion into  another   non-invariant vector by \enquote{differentiating}, i.e applying an operator of the form $1-k$ where the element $k$ is taken from a compact subset.

\begin{prop} \label{prop:noai}
There exists a constant $\alpha = \alpha(G) >0$ with the following property. Let $G\to \U(V)$ be any unitary representation and $v \in V$ any vector. If $k_0 \in K$ is an element satisfying
   \[   \|(1-k_0)v\| = \sup_{k\in K} \|(1-k)v\| \]
 then   
\[ \sup_{k\in K} \|(1-k)(1-k_0)v \| \geq \alpha    \|(1-k_0)v\|. \]
\end{prop}

\begin{proof}
Let $Q \subset G$ be the compact symmetric identity neighborhood and $\delta > 0$ be the constant  provided by Lemma~\ref{lem:ArAs}. Let $n \in \mathbb{N}$ be such that $Q \subset K^n$. Take  $\alpha=\delta/n$. 
Fix any element $k_0 \in K$ such that 
\[ c= \|(1-k_0)v \| = \sup_{k\in K} \|(1-k)v\|. \]
We set 
\[ u=(1-k_0)v \quad \text{and} \quad d=\sup_{k\in K} \|(1-k)u\|. \]
Note that $\|u\| = c$. Our goal is  to show that $d\geq \alpha c$. We will assume as we may that $c>0$.
 
Consider the complex-valued continuous function
\[ \phi \in C(G), \quad \phi(g)=\langle (1-g)v,u \rangle/c^2 \quad \forall g \in G. \]
Notice that $\|\phi|_{K}\|_\infty = \phi(k_0)= 1$. At this point we use the compact abelianization assumption together with Lemma \ref{lem:ArAs} to find a pair of elements $k_1\in K$ and $k_2 \in Q$
satisfying
    \[ |\phi(k_1)+\phi(k_2)-\phi(k_1k_2)|\geq \delta. \]
Denote $h_1=k_1^{-1} \in K$. We get
\begin{align*} n\alpha & =\delta \leq |\phi(k_1)+\phi(k_2)-\phi(k_1 k_2)|= |\langle (1-k_1)(1-k_2)v,u \rangle|/ c^2  \\
&=|\langle (1-k_2)v,(1-h_1)u \rangle|/c^2 \leq \|(1-k_2)v\|\cdot \|(1-h_1)u\|/c^2.\end{align*}
It follows from Lemma \ref{lemma:length function} that 
\[\|(1-k_2)v\| \le nc \quad \text{and} \quad \|(1-h_1)u\| \le d.\]
Putting everything together gives $n \alpha \le n d / c$. This inequality is equivalent to the desired conclusion.
\end{proof}


\subsection*{Constructing a discordant sequence of vectors}

Let $f \in \mathcal{A}(G)$ be any function.  
Denote $K = \mathrm{supp}(f) = \overline{\{g \in G \: : \: f(g) \neq 0\}}$. 
The following lemma, which is quite technical, will be used in the proof of Lemma~\ref{lem:mar3.7} below to allow for a smoothing procedure to be applied.

\begin{lemma} \label{lem:almostmax}
Consider the constants
$ \alpha_n = 1 - 2^{-n}$ for all  $n \in \mathbb{N}$.
Let $G \to \U(V)$ be any unitary representation.  Denote $p_n = P(\left[\alpha_n,\alpha_{n+1}\right])$ where $P$ is the projection-valued measure associated to $f$ regarded as a self-adjoint operator on $V$. Then every vector $v \in V$ satisfies for all $ n \ge 4$ that
\[ \sup_{k \in K} \|(1-k)p_nv\| \le 2 \sup_{k \in K} \|(1-k)fp_nv\|. \]
\end{lemma}

\begin{proof}
Let $v \in V$ be an arbitrary vector. We assume without loss of generality that $v \in p_n V$ and that $v \neq 0$. Set $u=f v$ and observe that $\|u\|\geq \alpha_n \|v\|>0$. We normalize these vectors so that  $\|u\|=1$. Note that $u \in p_n V$ so that 
\[\langle f u,u\rangle \leq \|f  u\|\leq \alpha_{n+1} \quad \text{and} \quad  \|(1-f)u\|\leq 1-\alpha_n = 2^{-n}.\]

Denote
\[ c=\sup_{k\in K} \|(1-k) u\|. \]
Every group element   $k\in K$ satisfies
\[ 2\Real\langle (1-k)u,u \rangle=2-2\Real\langle ku,u \rangle =\|(1-k)u\|^2\leq c^2. \]
As $K = \mathrm{supp}(f)$, we may act with the positive averaging operator $f$ and obtain
\[ 2\langle (1-f)u,u \rangle=2\Real\langle (1-f)u,u \rangle\leq c^2.\]
Putting all of the above information together gives
\[ \|(1-f)u\|\leq 2^{-n} = 2\cdot (1- \alpha_{n+1}) \leq 2\cdot (\langle u,u\rangle-\langle f u,u\rangle )= 2\cdot \langle (1-f) u,u\rangle \leq c^2. \]
Provided that   $n\geq 4$ we certainly have 
$ \|(1-f) u\| \leq 2^{-n} \leq 1/16$.
These last two inequalities give 
\[ \|(1-f)u\|^2 \leq c^2/16 \quad \text{so that} \quad 4\|(1-f)u\| \leq c. \]

Consider the vector $w=(1-f)v  \in p_nV$. 
We obtain
\[  \|w\| \leq \alpha_n^{-1} \|f w\| \le 2\|f w\| = 2\|f (1-f)v\|=2 \|(1-f)u\| \leq c/2. \]
This means that every element   $k\in K$  satisfies
\[ \|(1-k) v\|\leq \|(1-k) f v\|+\|(1-k)(1-f) v\|
\leq c+2\|w\|\leq 2c. \]
The above inequality is  the desired conclusion.
\end{proof}

We are ready to construct a discordant sequence of vectors, that is,  a sequence of vectors with the particular properties demanded in Lemma \ref{lem:mar3.7}, thereby extending Lemma IV.3.7 of \cite{margulis1991discrete} to products of groups with compact abelianization.

\begin{proof}[Proof of Lemma \ref{lem:mar3.7}]
\label{proof:proof of lemma 3.7}
Let $f_i \in \mathcal{A}(G_i)$  be a pair of continuous functions such that their supports $K_i = \mathrm{supp}(f_i)$ contain a neighborhood of the identity and generate the group $G_i$ for $i \in \{1,2\}$.
We take   $\alpha>0$ to be the constant   given in Proposition~\ref{prop:noai} with respect to the compact, symmetric and generating subset $K_2$ of the group $G_2$.
    
We regard $f_1$ and $f_2$ as a pair of commuting self-adjoint operators on the Hilbert space $V$. Let $P$ and $Q$  be the respective projection-valued measures. We consider the spectral projections
\[ p_n = P([\alpha_n,1]) \quad \text{and} \quad q_n = Q([\alpha_n,\alpha_{n+1}])\]
defined in terms of the constants $\alpha_n = 1 - 2^{-n}$ introduced in Lemma \ref{lem:almostmax}.
The   asymmetry is intentional.  Note that all of these projections pairwise commute. In addition, the $p_n$'s commute with $G_2$ and the $q_n$'s commute with $G_1$.
By the assumption that there are asymptotically $G$-invariant vectors and $V^{G_2}=0$,
we find two strictly increasing sequences $n_i,m_i \in \mathbb{N}$ with   $m_1\geq 4$ such that $p_{n_i}q_{m_i}\neq 0$.

We turn to constructing the desired sequence of vectors $u_i \in V$. For each  $i \in \mathbb{N}$ take an arbitrary  unit vector $v_i\in p_{n_i}q_{m_i}V$ and some element $k_i\in K_2$ such that 
   \[ \|(1-k_i)f_2 v_i\|=\sup_{k\in K_2} \|(1-k)f_2 v_i\|. \]
We set $u_i=(1-k_i)f_2 v_i$. Note that $u_i \neq 0$ for otherwise the non-zero vector $f_2 v_i$ would have been $G_2$-invariant.
By Proposition~\ref{prop:noai} and the choice of the constant $\alpha$ we get 
\[ \sup_{k\in K_2} \|(1-k)u_i\| \geq \alpha \|u_i\|. \]
This means that   the sequence $u_i$ is
not asymptotically $G_2$-invariant, as required.
As $v_i\in p_{n_i}V$ and as $f_1$ commutes with both $f_2$ and $k_i$, we see that $u_i\in p_{n_i}V$. Hence the sequence $u_i$ is asymptotically $G_1$-invariant.
The fact that it is $G_1$-uniform follows.
We are left to show that the sequence $u_i$ is $G_2$-uniform, which we now proceed to do.

Fix $\varepsilon>0$. We will  show that there exists an identity neighborhood $U_2\subset G_2$
such that for every element $g\in U_2$ and every $i$ we have 
    \[ \|(1-g)u_i\|\leq \varepsilon\|u_i\|. \]
The uniform continuity of the function $f_2$ together with the compactness of the subset $K_2$ allow us to find an identity neighborhood $U_2 \subset K_2 $
such that   every pair of elements $g\in U_2$ and $k\in K_2$ satisfy
    \[ \|(1-g)(1-k)f_2\|_\infty \leq \varepsilon/6. \]
Note that for every element $g \in U_2$ and all $i \in \mathbb{N}$ we have 
\[\supp (1-g)(1-k_i)f_2\subset K_2^3=K_2\cdot K_2\cdot K_2.\]
For every element $x \in K_2^3$, the two Lemmas \ref{lemma:length function} and \ref{lem:almostmax} respectively imply the first and second inequality in the following equation
\[ \|(1-x)v_i\| \leq 3\sup_{k\in K_2} \|(1-k) v_i\| \leq 6\|u_i\|.  \]
Since $m_G ((1-g)(1-k_i)f_2) = 0$, we have for every element $g\in U_2$ and all $i \in \mathbb{N}$ 
\[ (1-g)u_i=\int_G (1-g)(1-k_i)f_2(x)(x-1)v_i~ \mathrm{d}m_G(x). \]
Hence every element $g\in U_2$ satisfies 
\[ \|(1-g)u_i\| \leq \|(1-g)(1-k_i)f_2\|_\infty \cdot \sup_{x\in K_2^3} \|(x-1)v_i\|\leq \frac{\varepsilon}{6} \cdot 6\|u_i\|=\varepsilon  \|u_i\|
\]
for all  $i \in \mathbb{N}$. This means that the sequence $u_i$ is $G_2$-uniform, as required.
\end{proof}

We mention one additional lemma of Margulis, showing that an averaging operator can be applied to the \enquote{discordant sequence} constructed in Lemma \ref{lem:mar3.7} without losing its particular properties.

\begin{lemma}
\label{lemma:small smoothen preserves properties}
Let $G=G_1\times G_2$ where $G_1$ and $G_2$ are second countable locally compact groups.
 Let $G \to \U(V)$ be a unitary representation. Let $v_n \in \mathcal{H}$ be a \mbox{$G$-uniform}, asymptotically $G_1$-invariant and \emph{not} asymptotically $G_2$-invariant sequence of vectors. 
 
 Then there is an open identity neighborhood $U \subset G$ such that for every   function $\psi \in \mathcal{A}(G)$ with $\mathrm{supp}(\psi) \subset U$ the sequence $\psi v_n$ has the same three properties and satisfies $\|\psi v_n \| >  \|v_n\|/2$ for all $n \in \mathbb{N}$.
\end{lemma}

\begin{proof}
By  \cite[Lemma IV.3.6]{margulis1991discrete} we may find an identity neighborhood $U_1 \subset G$ such that for every function $\psi \in \mathcal{A}(G)$ with $\mathrm{supp}(\psi) \subset U_1$, the sequence $\psi v_n$ has the same above-mentioned three properties as the sequence $v_n$. As the sequence $v_n$ is $G$-uniform, there is another identity neighborhood $U_2 \subset G$ such that for every  function $\psi \in \mathcal{A}(G)$ with $\mathrm{supp}(\psi) \subset U_2$ we have $\|\psi v_n \| > \|v_n\|/2$ for all $n \in \mathbb{N}$. The desired conclusion follows by taking $U = U_1 \cap U_2$.
\end{proof}

\section{Koopman representations over actions with stabilizers}
\label{sec:Koopman reps}

Let $G$ be a second countable locally compact group acting continuously on a locally compact topological space $X$. Let $m$ be a $G$-invariant measure on the space $X$, either finite or infinite. We consider the corresponding Koopman representation $L^2(X)$, which is taken implicitly with respect to the measure $m$. The stabilizer map is the Borel measurable map given by
$$ \mathrm{Stab} : X \to \Sub{G}, \quad \mathrm{Stab} : x \mapsto \mathrm{Stab}_G(x) \quad \forall x \in X.$$

The following is the main result of this section. Roughly speaking, it says that under the right conditions, asymptotic invariance carries from one factor to the other. See Example \ref{example: dense projections non smooth} below regarding the necessity of some of its assumptions.

\begin{prop}
\label{prop:moving invariance around with dense projections}
Assume that $G = G_1 \times G_2$. Fix a function $\varphi \in \mathcal{A}(G)$. Let $u_n \in L^2(X)$ be a $G$-uniform and asymptotically $G_1$-invariant sequence of unit vectors of the form $u_n = \varphi v_n$ for some $v_n \in L^2(X)$ with $\|v_n\| \le 2$. If 
the sequence of probability measures $\mathrm{Stab}_*(|u_n|^2 \cdot m)$ converges to $\mu\in \Prob(\Sub G)$ and \mbox{$\mu$-almost} every subgroup has dense projections to the factor $G_2$ then the sequence $u_n$ is  asymptotically $G_2$-invariant.
\end{prop}

The proof of Proposition \ref{prop:moving invariance around with dense projections} will be given  after a bit of preliminary work.
 It will be very useful to introduce some notations first.

\begin{notation}
For every set $W\subset G$ we define the subset $\Omega(W) \subset X$ given by
\begin{align*}
\Omega(W)&=\{x\in X \: : \:  \Stab(x)\cap W \neq \emptyset \} = \\
&= \{ x \in X \: : \: \text{$\exists g \in W$ such that $g x = x$} \}.
\end{align*}
\end{notation}

Here are some elementary properties of the operation $\Omega$, whose proof is immediate.
For every subset $W \subset G$, family of subsets $W_i \subset G$ and element $g \in G$ we have
\begin{itemize}
\item $\Omega(W) = \Omega(W^{-1})$, 
\item $\Omega(\bigcup W_i)=\bigcup \Omega(W_i)$,
\item $\Omega(\bigcap W_i)\subset \bigcap_i\Omega(W_i)$, 
\item $g\Omega(W)=\Omega(gWg^{-1})$, and
\item if $e\in W$ then $\Omega(W)=X$.
\end{itemize}

\begin{notation}
For each measurable subset $Y \subset X$ and $f \in L^2(X)$ we denote
\[ \|f\|_Y^2 = \int_Y |f(x)|^2 \; \mathrm{d} m(x). \]
\end{notation}

\begin{remark}
In the following Lemma \ref{lem:newver} and in the proof of Proposition \ref{prop:moving invariance around with dense projections} we will be using notations such as $\psi \, |f|$. This notation stands for  the application of the smooth averaging operator $\psi$ to the vector $|f|$. See page \pageref{subsection:unitary reps} of \S\ref{sec:prelim}.
\end{remark}

\begin{lemma} \label{lem:newver}
    For every $\varphi\in C_c(G)$ and for every $\varepsilon>0$ there exist an identity neighborhood $U \subset G$ and a function $\psi\in L^1(G)$ satisfying $\psi\geq 0$ and $\|\psi\|_1 = \varepsilon$ with the following property ---    every vector $f\in L^2(X)$ and   element $g\in G$ satisfy
    \[ |(1-g) \varphi f(x)|\leq   \psi \, |f|(x). \]
    at $m$-almost every point  $x \in \Omega(Ug^{-1})$. In particular, for every measurable subset $Y\subset \Omega(Ug^{-1})$ we have
    \[ \|(1-g)\varphi f\|_Y \leq \| \, \psi|f|\, \|_Y. \]
\end{lemma}

\begin{proof}
Fix a function $\varphi\in C_c(G)$ and a constant $\varepsilon>0$.
Fix an arbitrary relatively compact identity neighborhood $U_0 \subset G$
and set $W=U_0\cdot\supp(\varphi)$.
By the uniform continuity of the function $\varphi$, it is possible to fix an identity neighborhood $U$ contained in $U_0$ such that for every pair of elements  $h,h'\in G$ with $h'h^{-1}\in U$,
\[ |\varphi(h')-\varphi(h)|< \varepsilon/m_G(W). \]
Set $\psi=\frac{\varepsilon}{m_G(W)}\cdot \chi_W$ and note that $\|\psi\|_1 = \varepsilon$.
For every $u\in U$, since the function $(1-u)\varphi$ vanishes outside $U\cdot\supp(\varphi)\subset W$, we get 
\[ |(1-u)\varphi| \leq \psi. \]
To conclude the proof, consider some vector  $f\in L^2(X)$ and some group element $g\in G$. For every point $x\in \Omega(Ug^{-1}) $ there exists  an element $u\in U$ such that $g^{-1}x=u^{-1}x$. Therefore 
\[ |(1-g) \varphi f(x)| = |(1-u) \varphi f(x)| 
\leq |(1-u) \varphi| \, |f|(x)
\leq \psi \, |f|(x).\]
The last conclusion follows by integration over $Y$.
\end{proof}


The following lemma is essentially an immediate consequence of the Portmanteau theorem on weak-$*$ convergence of probability measures. 

\begin{lemma}
\label{lemma:uniform density for sequences with dense projections}
Assume that $G = G_1 \times G_2$. Let $u_n \in L^2(X)$ be any sequence of unit vectors. Assume that the sequence of  probability measures $\mathrm{Stab}_*(|u_n|^2 \cdot m)$ converges to $\mu\in \Prob(\Sub G)$ and $\mu$-almost every subgroup has dense projections to $G_2$.
Then for every open subset $U \subset G_2$ and every $\varepsilon > 0$ there is an open relatively compact subset $V \subset G_1$   such that the subset $V \times U$
  satisfies
$$  \|u_n\|^2_{\Omega(V \times U)} \ge 1- \varepsilon $$
 for all $n$ sufficiently large.
\end{lemma}
\begin{proof}
Let $U \subset G_2$ be any open subset and let $\varepsilon > 0$.  Since $\mu$-almost every subgroup projects densely to $G_2$, there is some  open relatively compact subset $V \subset G_1$   such that the Chabauty open subset 
$$\widetilde{\Omega}(V \times U) =  \{ \Gamma \le G \: : \: \Gamma \cap (V \times U) \neq \emptyset \} \subset \Sub{G}$$
satisfies
$$ \mu(\widetilde{\Omega}(V \times U)) \ge 1 - \frac{\varepsilon}{2}.$$
By the  Portmanteau theorem $$\mathrm{Stab}_*(|u_n|^2 \cdot m)(\widetilde{\Omega}(V \times U))  \ge 1 - \varepsilon$$ 
for all $n$ sufficiently large. This is equivalent to the desired conclusion.
\end{proof}

We are ready to complete the main technical result of the current section \S\ref{sec:Koopman reps}.

\begin{proof}[Proof of Proposition \ref{prop:moving invariance around with dense projections}]
Assume towards contradiction that the sequence $u_n$ is not asymptotically $G_2$-invariant.
Using Lemma~\ref{lemma:not invariant means a particular element} we fix an element $g_2\in G_2$ such that 
\[ c= \limsup_n \|(1-g_2) u_n\|^2 > 0.\]
We will arrive at a contradiction   by showing that for every $n$ large enough
\begin{equation} \label{eq:inf}
\|(1-g_2) u_n\|^2 \leq \frac{c}{2}.
\end{equation}

We apply Lemma~\ref{lem:newver} with respect to the constant $\varepsilon=\sqrt{c/32}$  to get an identity neighborhood $U=U_1 \times U_2\subset G_1\times G_2$ and a function $\psi\in L^1(G)$ such that 
\begin{equation} \label{eq:psi}
    \|\psi\|_1 = \sqrt{c/32} 
\end{equation}
and such that for every function $f\in L^2(X)$, element $g\in G$ and measurable subset $Y\subset \Omega((U_1\times U_2)\cdot g^{-1})$,
\begin{equation} \label{eq:bound}
  \|(1-g) \varphi f \|_Y \leq \|\, \psi \, |f| \, \|_Y.
\end{equation}

Set $V_2=g_2U_2\cap U_2g_2$.
This is an open neighborhood of $g_2$ in $G_2$.
By Lemma~\ref{lemma:uniform density for sequences with dense projections}, applied to the sequence $u_n$ and the open set $V_2 \subset G_2$, there exists 
an open relatively compact subset $V_1 \subset G_1$ such that the subset
$ \Omega(V_1\times V_2)$
  satisfies
$$ \|  u_n \|^2_{\Omega(V_1\times V_2)}  \ge 1-\frac{c}{16} $$
for all $n$ sufficiently large.
Note that both subsets $\Omega(V_1\times g_2U_2)$ and $\Omega(V_1\times U_2g_2)$ contain the subset $\Omega(V_1\times V_2)$. Thus
we  have 
$$ \|u_n\|^2_{\Omega(V_1\times g_2U_2)}  \ge 1- \frac{c}{16}. $$
Using the left-invariance of the Haar measure and the equation \[ g_2^{-1}\Omega(V_1\times g_2U_2)=\Omega(V_1\times U_2g_2) \]
we also have
\[ \|g_2 u_n\|^2_{\Omega(V_1\times g_2U_2)}  = \|u_n\|^2 _{\Omega(V_1\times U_2g_2)} \ge 1- \frac{c}{16}. \]
Denote by $F$ the complement of $\Omega(V_1\times g_2 U_2 )$ in $X$.
We conclude that   $\|u_n\|^2_F\leq c/16$ as well as $\|g_2 u_n\|^2_F\leq c/16$ for all $n$ sufficiently large.
Therefore
\[ \|(1-g_2)u_n\|^2_F=\|u_n-g_2u_n\|^2_F \leq (\|u_n\|_F+\|g_2u_n\|_F)^2\leq (2\sqrt{\frac{c}{16}})^2=\frac{c}{4}\]
for all $n$ sufficiently large.
Note that 
\begin{align*}
\|(1-g_2)u_n\|^2 & = \|(1-g_2)u_n\|^2_{\Omega(V_1\times g_2U_2)}  + \|(1-g_2)u_n\|^2_F\\
& \leq  \|(1-g_2)u_n\|^2_{\Omega(V_1\times g_2U_2)} + \frac{c}{4}.
\end{align*}
Towards establishing Equation~\eqref{eq:inf}, we are therefore left to deal with the remaining summand and show that 
\begin{equation} \label{eq:Om}
\|(1-g_2)u_n\|^2_{\Omega(V_1\times g_2U_2)} \leq \frac{c}{4}
\end{equation}
for all  $n$ large enough. This is what we proceed to do.

We use the relative compactness of $V_1$ to find a finite collection of elements $h_1,\ldots,h_m\in G_1$ for some $m \in \mathbb{N}$ such that $V_1\subset \bigcup_{i=1}^m U_1h_i$.
Using the inclusion 
\[\Omega(V_1\times g_2U_2) \subset \Omega(\bigcup_{i=1}^m U_1h_i \times g_2U_2)=\bigcup_{i=1}^m\Omega(U_1h_i \times g_2U_2), \]
we fix an arbitrary measurable partition 
\[ \Omega(V_1\times g_2U_2)=\coprod_{i=1}^m A_i, \quad A_i\subset \Omega(U_1h_i\times g_2U_2). \]
For every $i \in \{1,\ldots,m\}$ there is the inclusion
\[ g_2^{-1}A_i\subset g_2^{-1}\Omega(U_1h_i\times g_2U_2)=\Omega(U_1h_i\times U_2g_2)=\Omega((U_1\times U_2)\cdot (h_ig_2)). \]
We apply Equation~\eqref{eq:bound} for $f=v_n$, $g=(h_ig_2)^{-1}$ and $Y=g_2^{-1}A_i$. This gives
\begin{align*}
\| (h_i^{-1}-g_2)u_n \|_{A_i}  &= 
\|(g_2^{-1}h_i^{-1}-1)u_n \| _{g_2^{-1}A_i} 
 = \| ((h_ig_2)^{-1} - 1)\varphi v_n \| _{g_2^{-1}A_i} \\
&\leq \| \, \psi \, |v_n| 
 \, \| _{g_2^{-1}A_i}.
\end{align*}
Combining this with Equation~\eqref{eq:psi} and the assumption that $\|v_n\|\leq 2$ gives
\begin{equation} \label{eq:sum}
\begin{split}
\sum_{i=1}^m \| (h_i^{-1}-g_2)u_n \|^2_{A_i} & \leq \sum_{i=1}^m \|\, \psi \, |v_n| \,\|^2_{g_2^{-1}A_i} \le \| \, \psi \, |v_n| \, \|^2 \\
& \leq \| \psi\|_1^2\cdot \|v_n\|^2 \leq \frac{c}{32} \cdot 4 = \frac{c}{8}.
\end{split}
\end{equation}
Making use of the elementary Lemma \ref{lemma on s and t} for each $i \in \{1,\ldots,m\}$  with respect to the vectors  $s=(1-h_i^{-1})u_n|_{A_i}$ and $t=(h_i^{-1}-g_2)u_n|_{A_i}$
we get  the estimates
\begin{equation} \label{eq:est}
\begin{split}
 \| (1-g_2)u_n \|^2_{A_i} & =\|(1-h_i^{-1})u_n+(h_i^{-1}-g_2)u_n \|^2_{A_i} \\
 & \le  6\cdot \|(1-h_i^{-1})u_n \|_{A_i} + \| (h_i^{-1}-g_2)u_n \|^2_{A_i}.
\end{split}
\end{equation}

Finally, 
we use the asymptotic $G_1$-invariance of the sequence $u_n$ to deduce that for every $n$ large enough 
\[
     \|(1-h_i^{-1})u_n\|_{A_i} \leq \|(1-h_i^{-1})u_n\| \leq  \frac{c}{48m}
\]
holds true for every $i\in\{1,\ldots,m\}$.
Combining Equations \eqref{eq:est} and  \eqref{eq:sum} we deduce that for sufficiently large $n$ we have
\begin{align*}
\| (1-g_2)u_n \|^2_{\Omega(V_1\times g_2U_2)} 
  & = \sum_{i=1}^m \| (1-g_2)u_n \|^2_{A_i} \\
  & \leq \sum_{i=1}^m\left( 6\cdot \|(1-h_i^{-1})u_n \|_{A_i}  + \| (h_i^{-1}-g_2)u_n \|^2_{A_i} \right) \\
  & \le \sum_{i=1}^m 6\cdot \frac{c}{48m} +  \sum_{i=1}^m \| (h_i^{-1}-g_2)u_n \|^2_{A_i}
   \leq \frac{c}{8}+ \frac{c}{8} = \frac{c}{4}.
\end{align*}
This gives Equation~\eqref{eq:Om}. The proof of the proposition is complete.
\end{proof}

\begin{lemma}
\label{lemma on s and t}
Let $V$ be a Hilbert space. If  $s,t \in V$ is a pair of vectors with $\|s\|,\|t\| \le 2$ then $\|s+t\|^2 \le 6 \|s\| + \|t\|^2$.
\end{lemma}
\begin{proof}
Indeed,
\[ \|s+t\|^2 \leq (\|s\|+\|t\|)^2=(\|s\|+2\|t\|)\cdot \|s\|+\|t\|^2
\leq 6 \|s\|+\|t\|^2. \]
\end{proof}

\subsection*{On the necessity of the assumptions}

Proposition~\ref{prop:moving invariance around with dense projections} has several technical assumptions. Ignoring these, it is tempting think of it loosely as the statement \enquote{an asymptotically $G_1$-invariant sequence that converges (in the sense of measures) to one with dense projections to $G_2$ is also asymptotically $G_2$-invariant}. However, this would be an oversimplification, as the following example shows.

\begin{eg}
\label{example: dense projections non smooth}
    Take $G_1=\mathbb{Z}$. Let  $G_2 = \widehat{\mathbb{Z}}$ be the profinite completion of $G_1$ and set $G=G_1\times G_2$. Consider the  group $\Gamma \cong \mathbb{Z}$ diagonally embedded in $G$.
    Thus $\Gamma$ is a uniform lattice in $G$ admitting dense projections to $G_2$. Let $m$ be the Haar probability measure on $G/\Gamma$.
Consider the Hilbert space $V=L^2_0(G/\Gamma)\cong L^2_0(\widehat{\mathbb{Z}})$.
As the group $\mathbb{Z}$ has no property $(\tau)$ it acts on its profinite completion without spectral gap \cite{lubotzky2005property}. Hence there exists an asymptotically $G_1$-invariant sequence  of  unit vectors $u_n \in V$.
As the group $G$ is abelian, we have  $\mathrm{Stab}_*(|u_n|^2 \cdot m) = \delta_\Gamma$ for all $n$.
However, the sequence $u_n$ is not asymptotically $G_2$-invariant, as the group $G_2$ is compact, and as such  has property (T).
\end{eg}

The group $\mathbb{Z}$ in Example \ref{example: dense projections non smooth} can be replaced with any other group not having  property $(\tau)$.

\subsection*{Spectral gap for actions of products} By putting together  our results on unitary representations of product groups, we  obtain a quite general spectral gap theorem. Recall that $G$ is a second countable locally compact group. Suppose in addition that $G$ is compactly generated.

\begin{theorem}[Spectral gap for actions of products]
\label{theorem:getting spectral gap - general case}
Assume that $G = G_1 \times G_2$ and that $G_2$ has a compact abelianization. Let $X$ be a locally compact topological $G$-space endowed with a $G$-invariant measure $m$, either finite or infinite. Assume that

\begin{itemize}
\item $L^2_0(X,m)^{G_2} = 0$, and
    \item 
For any asymptotically $G_1$-invariant sequence of unit vectors   $f_n \in L^2(X,m)$,  every  accumulation point $\mu \in \Prob(\Sub G)$ of the sequence of probability measures $\mathrm{Stab}_*(|f_n|^2 \cdot m)$  satisfies    $\overline{G_1 H} = G$ for $\mu$-almost   every subgroup  $H \in \Sub{G}$.
\end{itemize}
Then the Koopman $G$-representation $L^2_0(X,m)$ has spectral gap.
\end{theorem}
In a nutshell, Theorem \ref{theorem:getting spectral gap - general case} follows from the tension between Lemma~\ref{lem:mar3.7} and Proposition~\ref{prop:moving invariance around with dense projections}.
But of course, there are details to provide.  

\begin{proof}[Proof of Theorem \ref{theorem:getting spectral gap - general case}]
Assume towards contradiction that the Koopman representation $L^2_0(X,m)$ has no spectral gap. This representation has no $G_2$-invariant vectors by assumption.
According to  Lemma~\ref{lem:mar3.7} we may find   a \enquote{discordant} sequence of non-zero vectors $f_i\in L^2_0(X,m)$, namely, a sequence   with the following three properties:
\begin{enumerate}
    \item The sequence is $G$-uniform.
    \item The sequence is asymptotically $G_1$-invariant.
    \item The sequence is not asymptotically $G_2$-invariant.
\end{enumerate}

We fix a smoothing operator $\psi\in \mathcal{A}(G)$ as provided by Lemma~\ref{lemma:small smoothen preserves properties} with respect to the sequence $f_i$. Consider the sequence $g_i = \psi f_i$ which we may suppose (by scaling $f_i$) to be unit vectors. It continues to satisfy the above properties (1)-(3). Moreover it has the following two properties:
\begin{enumerate} \setcounter{enumi}{3}
 \item The sequence $g_i$ consists of unit  vectors.
 \item For each $i$ we have $g_i = \psi f_i$ and $\|f_i\|_2 \le 2$.
\end{enumerate}

Consider the sequence of probability measures $\mu_i = \mathrm{Stab}_*(|g_i|^2\cdot m)$ on the Chabauty space $\Sub{G}$. Up to passing to a subsequence, we may maintain all of the properties (1)-(5) and further assume that the probability measures $\mu_i$ converge in the weak-$*$ topology to some probability measure $\mu \in \mathrm{Prob}(\Sub{G})$ satisfying:
\begin{enumerate} \setcounter{enumi}{5}
\item The probability measure $\mu$ is $G_1$-invariant.
 \item $\mu$-almost every subgroup has dense projections to $G_2$.
\end{enumerate}

Finally, relying on the above properties, we are in a position to apply Proposition~\ref{prop:moving invariance around with dense projections} and deduce that the sequence $g_i$ must be asymptotically $G_2$-invariant. This is a contradiction.
\end{proof}

\begin{proof}[Proof of Theorem \ref{theorem:getting spectral gap} of the introduction]
The statement in the introduction is a special case of Theorem \ref{theorem:getting spectral gap - general case}. Indeed, in that statement the stabilizer of $m$-almost every point is contained in some given closed $G$-invariant subset $\Omega \subset \Sub{G}$ such that \emph{every} subgroup $H \in \Omega$ satisfies $\overline{G_1H } = G$.
\end{proof}

\begin{remark}
\label{remark:non-trivial part 2}
Strictly speaking, a  representation on the zero Hilbert space has no asymptotically invariant vectors, and as such has spectral gap.  For this reason, Theorem \ref{theorem:getting spectral gap - general case} is non-void only provided that the Hilbert space $L^2_0(X,m)$ is non-zero. In which case, it follows implicitly from the assumptions that both groups $G_1$ and $G_2$ are not trivial (so that $G$ is a direct product in a non-trivial manner).
\end{remark}

\section{Standard semisimple groups and irreducible lattices}
\label{sec:standard semisimple group}

In this section we set up our terminology, conventions and notations regarding semisimple groups and their lattices.
For brevity, we will use the unconventional notion of \enquote{standard semisimple groups}, which we now  introduce.

A \emph{standard simple group} is a topological group $G$ of the form 
  ${\bf G }(k)^+$, 
where $k$ is a  local field of zero characteristic and ${\bf G}$ is an isotropic adjoint connected absolutely simple $k$-algebraic group. The   topology on $G$ is the one induced from its $k$-analytic structure. Such groups are discussed in \cite[Chapter I, \S1.5, \S1.8 and \S2.3]{margulis1991discrete}.
In particular, the group $G$ is compactly generated, non-compact, and simple as an abstract group.
Moreover, the local field $k$ as well as the algebraic group ${\bf G}$ are canonically associated to $G$, in the sense that they can be recovered from its topological group
 structure.

 \begin{remark}
The requirement that ${\bf G}$ is absolutely simple  can be relaxed to simple, via a   restriction of scalars.
In particular, letting $k=\mathbb{R}$ or $k=\mathbb{Q}_p$ for some prime number $p$, and letting ${\bf G}$ be an isotropic adjoint connected simple $k$-algebraic group, we have that ${\bf G}(k)^+$ is a standard  simple group.
Every   simple real Lie group with trivial center is a standard  simple group 
 \cite[3.1.6]{zimmer2013ergodic}.
However, we emphasise that upon allowing this relaxation, we lose the possibility to recover $k$ and ${\bf G}$ from $G$,  
and also we will need to make several unpleasant adjustments to the discussion of arithmeticity and spectral gap below. So we keep with absolutely simple.
\end{remark}


The closure of the subfield $\mathbb{Q}$ in $k$ is a local field isomorphic to $\mathbb{Q}_p$ for some prime number $p$ or for $p=\infty$ (we use the convention $\mathbb{Q}_\infty=\mathbb{R}$).
We will say that $G$ is a \emph{standard simple group of type $p$}.
We denote $\mathrm{rank}(G) = \mathrm{rank}_{k}({\bf G})$.
This is a positive integer, by the assumption saying that $\bf G$ is isotropic.

A \emph{standard semisimple group} is a topological group   of the form $\prod_{i=1}^n G_i$ for some $n \in \mathbb{N}$, where each $G_i$ is a standard simple group. The groups $G_i$ are said to be the \emph{simple factors} of $G$.
We denote 
\[\mathrm{rank}(G) = \sum_{i=1}^n \mathrm{rank}(G_i).\]
If $\mathrm{rank}(G)=1$ then $G$ it is said to be of \emph{rank one}. Otherwise $\mathrm{rank}(G) \ge 2$ and $G$ is said to be of \emph{higher rank}.
For each prime number $p$ and for $p =\infty$, we define $G^{(p)}$ to be the subgroup of $G$ consisting of the product of all simple factors of type $p$. This is the \emph{$p$-component} of $G$. It is said to be a \emph{standard semisimple group of type $p$}.
We view $G^{(p)}$ as a subgroup of a $\mathbb{Q}_p$-algebraic group, by restricting the scalars of $\mathbf{G}_i$ for each simple factor $G_i$ of $G^{(p)}$ and taking the corresponding direct product. We restrict to $G^{(p)}$ the corresponding Zariski topology and denote it the \emph{$p$-Zariski topology} of $G^{(p)}$. 
In this way, we may also endow $G^{(p)}$ with a ${\mathbb{Q}_p}$-analytic structure, arising from it being a closed subgroup of the group of ${\mathbb{Q}_p}$-points of a ${\mathbb{Q}_p}$-algebraic group.

\begin{defn}
\label{def:fully Zariksi dense}
A subgroup $H \le G$ is  \emph{fully Zariski dense} if the projection of $H$ to each $p$-component $G^{(p)}$ is $p$-Zariski dense, for each prime number $p$ and for $p = \infty$.
\end{defn}

\subsection*{Lattices and invariant random subgroups}

Recall that any lattice in a standard semisimple group $G$ is fully Zariski dense by the Borel density theorem. This classical fact has been extended to invariant random subgroups, first  for real Lie groups in \cite{7s}, and then  for standard semisimple groups over a single local field in \cite{gelander2018invariant}. Relying on those results and ideas, we provide a statement applicable to any standard semisimple group.

\begin{prop}
\label{prop:BDT for IRS}
Let $G$ be a standard semisimple group and $\nu$ an ergodic invariant random subgroup of $G$. Then there are two semisimple factors (i.e. two normal subgroups) $M$ and $N$ of $G$ with $ N \cap M = \{e\}$ such that $\nu$-almost every subgroup 
\begin{enumerate}
    \item projects densely to each $p$-component $N^{(p)}$,
    \item projects discretely and fully Zariski densely to $M$, and
    \item is contained in $N \times M$ (i.e. projects trivially to $G/(N \times M)$).
\end{enumerate}
\end{prop}
\begin{proof}
Consider the $p$-component $G^{(p)}$ of the standard semisimple group $G$ for each prime number $p$ as well as for $p = \infty$. Let $\nu^{(p)}$ be the invariant random subgroup of the $p$-component $G^{(p)}$ obtained by considering the natural projection $\pi^{(p)} : G \to G^{(p)}$ and pushing forward $\nu$ via the map $H \mapsto \overline{\pi^{(p)}(H)}$. We may apply the Borel density theorem for invariant random subgroups of standard semisimple groups over a single local field \cite[Theorem 1.9]{gelander2018invariant} to find a pair of normal subgroups $N^{(p)},M^{(p)} \le G^{(p)}$ with $N^{(p)} \cap M^{(p)} = \{e\}$ and such that $\nu^{(p)}$-almost every subgroup contains $N^{(p)}$, projects discretely and $p$-Zariski densely to $M^{(p)}$ and is contained in $N^{(p)} \times M^{(p)}$. The two desired normal subgroups are constructed by taking $N = \prod_{p} N^{(p)}$ and $M = \prod_p M^{(p)}$ where $p$ runs over $p=\infty$ and  all prime numbers   involved in $G$.
\end{proof}

It will be useful for us  to know the following property of lattices in standard semisimple groups.

\begin{lemma}
\label{lemma:lattices in simple analytic groups are ICC}
Let $G$ be a standard semisimple group and $\Gamma$ a lattice in $G$. Then the conjugacy class of every non-trivial element $\gamma \in \Gamma$ is infinite.
\end{lemma}
\begin{proof}
 Let $\gamma \in \Gamma$ be some non-trivial element and assume towards contradiction that its centralizer $\Delta = C_\Gamma(\gamma)$ has a finite index in $\Gamma$. It follows that $\Delta$ is a lattice in $G$. Therefore the projection of $\Delta$ to each $p$-component $G^{(p)}$ is $p$-Zariski-dense by the Borel density theorem. It follows that the projection of the element $\gamma$ to each $p$-component $G^{(p)}$ must be central. We arrive at a contradiction to the fact that   the standard semisimple group $G$ is center-free.
\end{proof}

Lastly, we briefly recall the notion of irreducible lattices.

\begin{defn}
\label{def:irreducible lattice}
 A lattice $\Gamma$ in a standard semisimple group $G$ is said to be \emph{irreducible} if the projection of $\Gamma$ to $G/H$ has a dense image for each simple factor $H$ of $G$.
\end{defn}

Indeed, a lattice $\Gamma$ in a standard semisimple group $G$ is irreducible if and only if the Borel $G$-space $G/\Gamma$ with the normalized probability measure is irreducible in the sense of \S\ref{sec:prelim}. If $G$ is a standard simple group, then all lattices in $G$ are irreducible, as the condition in the above Definition \ref{def:irreducible lattice}  is satisfied trivially. For lattice subgroups, the classical notion of  irreducibility  coincides with the notion of  irreducible subgroups (i.e. dense projections) defined in the introduction.

\subsection*{Arithmetic groups}
We now describe the construction of the standard semisimple group associated 
with a given algebraic group defined over a number field.

\begin{eg} \label{ex:H}
Fix a number field $K$ and an adjoint connected absolutely simple $K$-group ${\bf H}$.
Recall that a \emph{place} $s:K\to k_s$ is a dense embedding of the field $K$ in a local field $k_s$, defined up to a natural equivalence.
Corresponding to $s$ we get the $k_s$-group ${\bf H}_s$ obtained by an extension of scalars. This is an adjoint connected absolutely simple $k_s$-algebraic group. We say that the place $s$ is isotropic if the group ${\bf H}_s$ is $k_s$-isotropic, and in this case we denote $H_s={\bf H}_s(k_s)^+$.
This is the standard simple group associated with ${\bf H}$ at the isotropic place $s$.
Fixing a finite set $S$  of isotropic places, we obtain the standard semisimple group associated with ${\bf H}$ at $S$, namely $H_S= \prod_{s\in S} H_s$.
Note that the group ${\bf H}(K)$ embeds diagonally in $\prod_{s\in S} {\bf H}_s(k_s)$. 
Accordingly, as $H_s$ has finite index in ${\bf H}_s(k_s)$ for each $s \in S$, an appropriate  finite-index subgroup of ${\bf H}(K)$ is embedded in the finite-index subgroup $H_S$ of $\prod_{s\in S} {\bf H}_s(k_s)$.
This embedding is dense by the strong approximation theorem.
\end{eg}

Recall that a pair of  subgroups $\Delta_1$ and $\Delta_2$ of a standard semisimple group $G$ are called \emph{commensurable} if their intersection $\Delta_1 \cap \Delta_2$  has finite index in both $\Delta_1$ and $\Delta_2$. 
Commensurability is an equivalence relation.
Being a lattice is obviously a commensurability invariant. Likewise, being an irreducible lattice in a standard semisimple group 
%
is also a commensurability invariant.

The groups discussed in Example~\ref{ex:H} admit a canonical commensurability class of irreducible lattices, described in the following paragraph.

\begin{eg} \label{ex:Gamma_S}
Fix a number field $K$ and an adjoint connected absolutely simple $K$-group ${\bf H}$.
A place $s:K\to k_s$ is called Archimedean (or non-Archimedean) if the local field $k_s$ is.
Let $\mathcal{O}<K$ be the ring of integers. For every non-Archimedean place $s$ denote by $\pi_s \lhd \mathcal{O}$ the corresponding prime ideal, that is, the preimage of unique maximal ideal in the closure of $s(\mathcal{O})$.
Fix a finite set of isotropic places $S$ which includes all Archimedean isotropic places.
Let $\mathcal{O}_S$ be the localization of $\mathcal{O}$ by the ideals $\pi_s$ ranging over all the non-Archimedean places $s\in S$.
Consider the standard semisimple group $H_S$ discussed in Example~\ref{ex:H}.
Upon choosing a non-trivial $K$-representation $\rho:{\bf H}\to \GL_n$,
we set 
$\Lambda_S=\rho^{-1}(\GL_n(\mathcal{O}_S))$.
Note that the commensurability class of $\Lambda_S$ does not depend on the particular choice of $\rho$.
We consider the finite-index subgroup of ${\bf H}(K)$ which embeds densely in $H_S$, as discussed at the end of Example~\ref{ex:H},
and  let $\Lambda_S^+$ be its intersection with $\Lambda_S$.
This is a finite-index subgroup of $\Lambda_S$ which embeds in $H_S$. We let $\Gamma_S$ be the corresponding image of $\Lambda_S^+$ is $H_S$.
Then $\Gamma_S$ is an irreducible lattice in the standard semisimple group $H_S$.
Its commensurability class is independent on the chosen representation $\rho$.
\end{eg}

\begin{defn}
\label{defn:arithmetic lattice}
An irreducible lattice $\Gamma$ in a standard semisimple group $G$ is said to be \emph{arithmetic} if there exist
a number field $K$, an adjoint connected absolutely simple $K$-group ${\bf H}$
and a finite set of isotropic places $S$ which includes all Archimedean isotropic places, 
such that $G$ is isomorphic as a topological group to the group $H_S$ discussed in Example~\ref{ex:H} and the image of $\Gamma$ under this isomorphism can be conjugated to a lattice in the commensurability class of the lattice $\Gamma_S$ described in Example~\ref{ex:Gamma_S}.
\end{defn}

The following is a fundamental theorem due to Margulis.

\begin{theorem}[Margulis Arithmeticity, {\cite[Chapter IX, Theorem 1.11]{margulis1991discrete}}] \label{thm:arithmeticity}
    Let $G$ be a standard semisimple group of higher rank and  $\Gamma<G$ be an irreducible lattice.
    Then $\Gamma$ is arithmetic.
\end{theorem}

\subsection*{Spectral gap}
We end this section by stating an important corollary of Clozel's theorem on  spectral gap, namely \cite[Theorem 3.1]{Clozel}.

\begin{theorem}
\label{thm:Clozel}
Let $G$ be a standard semisimple group and  $\Gamma<G$ be an irreducible lattice.
Let $F$ be a simple factor of $G$. Then the unitary $F$-representation on the Hilbert space $L^2_0(G/\Gamma)$ has a spectral gap.
\end{theorem}

\begin{proof}
If the standard semisimple group $G$ is simple then it certainly has no proper simple factors. Therefore the desired conclusion follows from   \cite[Lemma 3]{bekka1998unique} for Lie groups, \cite[Theorem 1]{bekka2011lattices} for standard simple groups over non-Archimedean local fields,    \cite[Chapter III, Corollary 1.10]{margulis1991discrete} for uniform lattices  and \cite[Theorem 1.8]{gelander2022effective} in general.

From now on,  assume that the standard semisimple group $G$ is not simple, and fix a simple factor $F$. In particular $G$ is of higher rank.
By Margulis Arithmeticity (Theorem \ref{thm:arithmeticity}) the lattice $\Gamma$ is arithmetic.
We adopt below the notation introduced in Example~\ref{ex:Gamma_S} and Definition \ref{defn:arithmetic lattice}. In particular, there is a number field $K$, an adjoint connected absolutely simple $K$-group ${\bf H}$ and a finite set $S$ of isotropic places containing all the Archimedean isotropic ones, such that  $G$ is isomorphic as a topological group to the group $H_S$. The simple factor $F$ corresponds to the factor $H_s$ for some place $ s \in S$, as discussed in Example~\ref{ex:H}. A conjugate of $\Gamma$ is commensurable to the lattice $\Gamma_S$.
Up to replacing $\Gamma$ by this conjugate, we will assume that $\Gamma$ is actually commensurable to $\Gamma_S$. This does not change the unitary $H_s$-representation $L^2_0(H_S/\Gamma)$ (up to an isomorphism).
By \cite[Lemma 3.1]{Kleinbock-Margulis}, we know that the $H_s$-representation $L^2_0(H_S/\Gamma)$ has spectral gap if and only if the $H_s$-representation $L^2_0(H_S/\Gamma_S)$ does.
It remains for us  to prove the latter statement.

Let $\tilde{\bf H}$ be the simply connected covering of $\bf H$ and $\pi:\tilde{\bf H}\to {\bf H}$ be the associated $K$-central isogeny.
By \cite[Theorem 3.1]{Clozel} the trivial representation is isolated in the automorphic dual of $\tilde{H}_s=\tilde{\bf H}_s(k_s)$.
This means that the unitary $\tilde{H}_s$-representation $L^2_0(\tilde{\bf H}(\mathbb{A}_{K})/\tilde{\bf H}(K))$ has a spectral gap, where $\mathbb{A}_{K}$ is the $K$-adele ring.
Write $\tilde{\bf H}(\mathbb{A}_{K})=\tilde{H}_S \times \tilde{H}_{S^c}$, where $\tilde{H}_S$ is the product of groups corresponding to the places in $S$ and $\tilde{H}_{S^c}$ the restricted product of groups corresponding to the complement set of $S$. Since $S$ contains all isotropic Archimedean factors, we can find a compact open subgroup $C<\tilde{H}_{S^c}$.
Then $\tilde{H}_S \times C$ is an open subgroup of $\tilde{\bf H}(\mathbb{A}_{K})$ and $\Lambda=\tilde{\bf H}(K)\cap (\tilde{H}_S \times C)$ is a lattice in $\tilde{H}_S \times C$.
We identify $(\tilde{H}_S \times C)/\Lambda$ with an $\tilde{H}_s$-invariant open subset of $\tilde{\bf H}(\mathbb{A}_{K})/\tilde{\bf H}(K)$, and view $L^2_0((\tilde{H}_S \times C)/\Lambda)$ as an $\tilde{H}_s$-subrepresentation of $L^2_0(\tilde{\bf H}(\mathbb{A}_{K})/\tilde{\bf H}(K))$. 
We conclude that $L^2_0((\tilde{H}_S \times C)/\Lambda)$ has a spectral gap regarded as a $\tilde{H}_s$-representation, and so does the $\tilde{H}_s$-subrepresentation of $C$-invariants, denoted $L^2_0((\tilde{H}_S \times C)/\Lambda)^C$. We identify this representation with $L^2_0(\tilde{H}_S/\Lambda_1)$, where $\Lambda_1$ denotes the image of $\Lambda$ under the proper projection map $\tilde{H}_S \times C \to \tilde{H}_S$.

Next, by \cite[Proposition I.1.5.5 and Theorem I.2.3.1]{margulis1991discrete} the central isogeny $\pi$ gives a surjective map $\tilde{H}_S\to H_S$ with a finite central kernel, which we denote by $Z = \ker \pi$.
Let $\Lambda_2<H_S$ be the image of $\Lambda_1$ under this map and identify the unitary representation $L^2_0(H_S/\Lambda_2)$ with the space of $Z$-invariants in $L^2_0(\tilde{H}_S/\Lambda_1)$. We conclude that the $H_s$-representation $L^2_0(H_S/\Lambda_2)$ has a spectral gap.
Since, by construction, $\Lambda_2$ is commensurable with $\Gamma_S$, and using   \cite[Lemma 3.1]{Kleinbock-Margulis} once more, we conclude that the $H_s$-representation $L^2_0(H_S/\Gamma_S)$ has a spectral gap.
This finishes the proof.
\end{proof}

\section{Discrete subgroups of standard semisimple groups}
\label{sec:geometry of discrete subgroups of products}

Throughout this section $G$ will denote a standard semisimple group, a notion introduced and discussed in detail in \S\ref{sec:standard semisimple group}. Terminology and notation from \S\ref{sec:standard semisimple group} will be used   freely in the current \S\ref{sec:geometry of discrete subgroups of products}.

\begin{remark}
For readers interested primarily in the case of semisimple real Lie groups, we remark that a center-free connected semisimple real Lie group without compact factors is a standard semisimple group of type $\infty$.
\end{remark}








\subsection*{Deformations and  Zariski density}

Recall that $\mathbb{Q}_p$ stands either for the field of $p$-adic numbers when $p$ is a prime number or for the field $\mathbb{R}$ when $p$ is $\infty$. We use $\mathrm{Ad}$ to denote the adjoint representation.

\begin{lemma}
\label{lemma:Z-dense using algebra}
Let $G$ be a standard semisimple group of type $p$, where $p$ is either $\infty$ or a prime number.   A closed subgroup $H \le G$ is $p$-Zariski-dense if and only if the projection of $H$ to each simple factor of $G$ is infinite and 
\[ \mathrm{span}_{\mathbb{Q}_p} \mathrm{Ad}(H) = \mathrm{span}_{\mathbb{Q}_p} \mathrm{Ad}(G). \]
\end{lemma}
\begin{proof}
To ease our notation denote  $k = \mathbb{Q}_p$. Let $\mathfrak{g}$ denote the semisimple $k$-Lie algebra of     $G$ where the latter is regarded as a $k$-analytic group. Set $A = \mathrm{span}_{k} \mathrm{Ad}(G)$,
which is an associative $k$-subalgebra 
of $\mathrm{End}_k(\mathfrak{g})$. Consider a closed subgroup $H \le G$ 
such that   $\mathrm{Ad}(H)$ spans the $k$-algebra $A$  and such that $H$ has an infinite projection to each simple factor of $G$. Let $\mathfrak{h}$ denote the Lie algebra of the $p$-Zariski closure of $H$. Certainly $\mathfrak{h}$ is an  $\mathrm{Ad}(H)$-invariant  subalgebra of $\mathfrak{g}$. The projection of the subalgebra $\mathfrak{h}$ to each simple ideal of $\mathfrak{g}$ is non-zero. Hence $\mathfrak{h} = \mathfrak{g} $ and so the $p$-Zariski closure of the subgroup $H$ is  open (in the Hausdorff topology arising from the $k$-analytic structure). Therefore $H$ is $p$-Zariski dense by \cite[Lemma 3.2]{platonov1993algebraic}, as required.  The converse direction of the lemma is immediate.
\end{proof}

\begin{lemma}\label{lem:Z-dense}
Let $G$ be a standard semisimple group of type $p$, where $p$ is either $\infty$ or a prime number.    The subset of $\Sub{G}$ consisting of $p$-Zariski-dense subgroups is open in the Chabauty topology.
\end{lemma}

\begin{proof}
Denote $k = \mathbb{Q}_p$ and  set $A = \mathrm{span}_k \mathrm{Ad}(G)$ as in the proof of Lemma \ref{lemma:Z-dense using algebra}. 
Let $H \le G $ be a closed $p$-Zariski-dense subgroup. 
We know by Lemma \ref{lemma:Z-dense using algebra} that $\mathrm{Ad}(H)$ spans $A$ over $k$. Thus we may find a finite set of elements $h_1,\ldots,h_m \in H$ for some $m \in \mathbb{N}$ such that $\mathrm{span}_k \{ \mathrm{Ad}(h_i) \} = A$.   It follows that  there is some    Chabauty neighborhood of $H$ such that any group $H'$ in that neighborhood    has   $\mathrm{span}_k \text{Ad}(H') = A$.  

In order to deduce that every such subgroup $H'$ is $p$-Zariski-dense,    it is enough by Lemma \ref{lemma:Z-dense using algebra} to show that the projection of $H'$ to any simple factor of $G$ is infinite. In doing so, we may pass to smaller Chabauty neighborhood of the subgroup $H$, and assume 
  that $H'$ lies in that neighborhood.

In the Archimedean case one may argue as follows. By the Jordan--Schur theorem \cite{jordan1878memoire}, there is some $n \in \mathbb{N}$ such that every finite subgroup of $G$ admits a normal abelian subgroup of index at most $n$ (for a more recent treatment see e.g. \cite[Theorem 8.29]{raghunathan1972discrete}). Since $H$ is real-Zariski-dense, there are elements $\alpha,\beta\in H$ such that the projection of $[\alpha^{n!},\beta^{n!}]$ to every simple factor of $G$ is non-trivial. This property is preserved in a small Chabauty neighborhood of $H$, and guarantees that the projection to each factor is infinite.

In the non-Archimedean case the above claim is straightforward, since there is an upper bound on the order of  finite subgroups \cite[Theorem 1 on p. 124]{serre2009lie}. 
\end{proof}

Using the above Lemma \ref{lem:Z-dense} combined with the definition of the Chabauty topology we obtain the following.

\begin{cor}\label{cor:Z-dense-ss}
Let $G$ be a standard semisimple group. The subset  of $\Sub{G} $ consisting of fully Zariski dense subgroups (in the sense of Definition \ref{def:fully Zariksi dense})   is open in the Chabauty topology.
\end{cor}


\subsection*{Subgroups of product groups}

We assume for the remainder of \S\ref{sec:geometry of discrete subgroups of products} that the  group $G$ is of the form $G = G_1 \times G_2$ where $G_1$ and $G_2$ are both non-trivial standard semisimple   subgroups.   Let $\mathrm{pr}_i : G \to G_i$ for $ i\in\{1,2\}$ denote the natural projections.

\begin{lemma}
\label{lemma:no conjugate limit has Zariski dense intersections}
Let $\Delta$ be a discrete subgroup of $G$.    If $\Lambda \in \overline{\Delta^G}$ is a discrete conjugate limit  whose intersection with $G_2$ is fully Zariski dense   then   $\Lambda \cap G_1$ is a subgroup of some conjugate limit of $\Delta \cap G_1$.
\end{lemma}

The terminology \emph{conjugate limit} was introduced in Definition \ref{def:conjugate limit}.

\begin{proof}[Proof of Lemma \ref{lemma:no conjugate limit has Zariski dense intersections}]
Let $\Lambda$ be a discrete conjugate limit of $\Delta$, say $\Lambda = \lim_j \Delta^{g_j}$ in the Chabauty topology for some sequence of elements $g_j \in G$.
Since $\Lambda$ is discrete there is some identity neighborhood $U \subset G$ such that $\Delta^{g_j}\cap U=\{e\}$ for all $j \in \mathbb{N}$.

Assume that the subgroup $\Lambda_2 = \Lambda \cap G_2$ is fully Zariski dense in $G_2$. Corollary \ref{cor:Z-dense-ss} allows us to find a finite subset $S \subset \Lambda_2$ such that any sufficiently small deformation of the set $S$ inside $G_2$ generates a subgroup  which is still fully Zariski dense in $G_2$. For all sufficiently large $j$  we find finite subsets $R_{j} \subset \Delta^{g_j}$  such that the projection to $G_2$  of the group generated by $R_{j}$ is fully Zariski-dense  and  $R_{j}$ converge to $ S$ in the obvious sense.

Consider any particular element  $h \in \Lambda \cap G_1$. There are elements $h_j \in \Delta^{g_j}$ such that the sequence $h_j$ converges to the element $h$.
Note that for all $j$ sufficiently large we have $\left[R_j,h_j\right] \subset U \cap \Delta^{g_j} = \{e\}$. This  implies that $h_j \in C_{G}(\left<R_j\right>)$.  Since the projection of the subgroup $\left<R_j\right>$ to $G_2$ is fully Zariski-dense, it follows that $h_j \in (\Delta^{g_j} \cap G_1) = (\Delta \cap G_1)^{g_j}$ for all $j$ sufficiently large. This shows that the intersection $\Lambda \cap G_1$ is contained in any conjugate limit of the intersection $\Delta \cap G_1$ obtained as an accumulation point of the sequence $(\Delta \cap G_1)^{g_j}$, as required.
\end{proof}

\begin{cor}
\label{cor:special cases of conjugate limit with Zariski dense intersection}
Let $\Delta$ be a discrete subgroup of $G$.   Let  $\Lambda \in \overline{\Delta^G}$ be a discrete conjugate limit of $\Delta$ whose intersection with $G_2$ is fully Zariski dense.  
\begin{enumerate}
    \item If $\Delta \cap G_1$ is trivial then $\Lambda \cap G_1$ is trivial.
    \item If $\Delta \cap G_1$ is not fully Zariski dense in $G_1$ then $\Lambda \cap G_1$ is not fully Zariski dense in $G_1$.
\end{enumerate}
\end{cor}
\begin{proof}
Both parts of the statement follow from the previous Lemma \ref{lemma:no conjugate limit has Zariski dense intersections}. Indeed, part (1) follows immediately. To deduce part (2) we need to further rely on the fact that  being fully Zariski dense is a Chabauty open condition; see Corollary \ref{cor:Z-dense-ss}.
\end{proof}

\subsection*{Random subgroups invariant for a single factor}

Recall that $G$ is a  standard semisimple group  which is  assumed to be a direct product $G = G_1 \times G_2$ of two standard semisimple factors with projections $\mathrm{pr}_i : G \to G_i$.

\begin{lemma}
\label{lemma:types of intersection}
Let $\Delta_1 \le G_1$ be a closed, not relatively compact and fully Zariski dense subgroup.
 Let $\nu$ be a $\Delta_1$-invariant probability measure on $\mathrm{Sub}(G)$ such that $\nu$-almost every subgroup  intersects $G_1$ trivially and projects to $G_2$ discretely.  Then $\nu$-almost every subgroup is contained in $G_2$.
\end{lemma}

This is a variant of \cite[Lemma 3.14]{fraczyk2023infinite}.
The proof is very short and we reproduce it here.
\begin{proof}[Proof of Lemma \ref{lemma:types of intersection}]
We may suppose without loss generality that the measure $\nu$ is $\Delta_1$-ergodic. 
Recall that $\mathrm{pr}_i$ denotes the projection from the group $G$ to each semisimple factor $G_i$. 
Assume towards contradiction that not $\nu$-almost every subgroup is  contained in $G_2$. Then one can find an open 
subset $O_2 \subset G_2 \setminus \{e_2\}$ with respect to which
$$ \nu( \{ H \le G \: : \: |\mathrm{pr}_1(H \cap \mathrm{pr}_2^{-1}(O_2))\setminus \{e_1\}| = 1 \} )  > 0. 
$$
This set is  $\nu$-conull  by $\Delta_1$-ergodicity.
The push forward of $\nu$ via the map taking a subgroup $H $ to the unique point in $\mathrm{pr}_1(H \cap \mathrm{pr}_2^{-1}(O_2))$ gives a probability measure on $G_1\setminus \{e_1\}$ which is invariant under  conjugation by $\Delta_1$. By projecting to one of the simple factors of the group $G_1$, we may assume without loss of generality that $G_1$ is simple. The existence of such a probability measure stands in contradiction to  \cite[Proposition 1.9]{bader2017almost} (see also the main result of \cite{shalom1999invariant}).
\end{proof}

In the context of semisimple Lie groups, another way to arrive at a contradiction  would be to borrow the argument from \cite[p. 38]{furman2001mostow}.
Yet another such way is  the closely related to \cite[Corollary 3.11]{fraczyk2023infinite} dealing with stationary random subgroups.



\begin{rmk}
The  assumption that a Zariski-dense subgroup is   not relatively compact is only relevant in the non-Archimedean case. Indeed, in the real case, any Zariski-dense subgroup of an isotropic simple algebraic group is automatically not relatively compact.    
\end{rmk}

\begin{prop}
\label{prop:getting dense projection to G2}
Let $G$ be a standard semisimple group of the form $G = G_1 \times G_2$ where $G_1$ is standard semisimple and $G_2$ is standard \emph{simple} of type $p_0$ for $p_0=\infty$ or for some prime number $p_0$.   Let $\nu$ be a $G_1$-invariant Borel probability measure on $\Sub{G}$.  If   $\nu$-almost every subgroup $\Gamma$ satisfies that
\begin{enumerate}
    \item $\Gamma$ is discrete,
    \item $\Gamma$ is not contained in the factor $G_2$,
    \item $\Gamma$ has $p_0$-Zariski-dense and not relatively compact projection to  $G_2$, and
    \item if $\Gamma \cap G_2$ is $p_0$-Zariski-dense then $\Gamma \cap G_1$ is trivial
\end{enumerate}
then $\nu$-almost every subgroup $\Gamma$ has a dense projection to  $G_2$.
\end{prop}

\begin{proof}
By passing to  ergodic components we may assume without loss of generality that the measure $\nu$ is $G_1$-ergodic.  
The $G_1$-invariant map $\sub(G)\to \sub(G_2)$ given by  $\Gamma \mapsto \overline{\mathrm{pr}_2(\Gamma)}$ is therefore $\nu$-essentially constant.
We denote its essential image by $\Delta_2 \in \Sub{G_2}$. Note that   $\Delta_2$ is a closed, non-compact and $p_0$-Zariski-dense   subgroup of $G_2$. Since the group $G_2$ is simple, it must either be the case that $\Delta_2 = G_2$ or that $\Delta_2$ is discrete; see e.g. \cite[\S3]{gelander2013dynamics}. In the first case we are done. In what follows we assume towards contradiction that the subgroup $\Delta_2$ is discrete.



We apply the Borel density theorem for invariant random subgroups (Proposition \ref{prop:BDT for IRS}) with respect to the pushforward of $\nu$ via the map $\Sub{G}\to\Sub{G_1}$ given by $\Gamma \mapsto \overline{\mathrm{pr}_1(\Gamma)}$. This provides a pair of normal subgroups $N,M \le G_1$ such that $N \cap M =\{e\}$ and such that $\nu$-almost every subgroup $\Gamma$ projects densely to each $p$-component $N^{(p)}$, projects discretely and fully Zariski-densely to $M$ and is contained in $N \times M$.
From this point onward, we assume as we may  that $G_1=M\times N$.


The measure $\nu$ is $\Delta_2$-invariant by \cite[Lemma 7.2]{fraczyk2023infinite}. 
We fix a generic $\Delta_2$-ergodic component $\nu_0$ of the measure $\nu$.
By $\Delta_2$-ergodicty, there is some fixed non-trivial closed subgroup $\Delta_1 \le G_1$  such that $\overline{\mathrm{pr}_1(\Gamma)} = \Delta_1$ for $\nu_0$-almost every subgroup $\Gamma$. 
As $\nu_0$ is generic,
we have that  $\Delta_1$ is fully Zariski-dense in $G_1$ and not relatively compact.
The measure $\nu_0$ is $\Delta_1$-invariant by \cite[Lemma 7.2]{fraczyk2023infinite} and $\Delta_1$-ergodic by \cite[Corollary 7.3]{fraczyk2023infinite}. Thus, $\nu_0$ can be regarded as a discrete   $\Delta_1$-ergodic and $\Delta_2$-ergodic invariant random subgroup of the  product group $\Delta_1 \times \Delta_2$. 

We now claim that the  factor $N$ must be   trivial, so that $G_1=M$ and the subgroup $\Delta_1$ is discrete.
We consider the map $\sub(G)\to \sub(N)$ given by $\Gamma\mapsto \Gamma\cap N$.
By \mbox{$\Delta_2$-ergodicity} this map is $\nu_0$-essentially constant.
Its essential image is a certain discrete subgroup of $N$ normalized by $\Delta_1$,   as the measure $\nu_0$ is $\Delta_1$-invariant. Since $\Delta_1$ projects densely to each $p$-component $N^{(p)}$, this discrete subgroup  must be trivial.
Consider the product decomposition $G=N\times (M\times G_2)$. We have established that $\nu_0$-almost every subgroup intersects $N$ trivially and projects to $M\times G_2$ discretely.
Since $\nu_0$ is $\Delta_1$-invariant, we may apply Lemma~\ref{lemma:types of intersection} with respect to this particular product decomposition, and deduce that $\nu_0$-almost every subgroup is contained in $M\times G_2$. This  proves the claim.

We consider the map $\sub(G)\to \sub(G_1)$ given by $\Gamma\mapsto \Gamma\cap G_1$.
By $\Delta_2$-ergodicity this map is $\nu_0$-essentially constant
and we denote its essential image by $\Lambda_1$.
Note that $\Lambda_1 \lhd \Delta_1$ as the measure $\nu_0$ is $\Delta_1$-invariant.
By Lemma~\ref{lemma:types of intersection} applied with respect to the product decomposition $G = G_1 \times G_2 = M \times G_2$  we have that $\Lambda_1$ is non-trivial.

Now we invert   the roles of the two factors $G_1$ and $G_2$ and consider the  map $\Sub{G} \to \Sub{G_2}$ given by $\Gamma\mapsto \Gamma\cap G_2$. This map   $\nu_0$-essentially  constant and its  essential image is some  normal subgroup $\Lambda_2\lhd \Delta_2$. The subgroup $\Lambda_2$ is non-trivial by yet another application of  Lemma~\ref{lemma:types of intersection}. Hence   $\Lambda_2$ is $p_0$-Zariski dense in $G_2$, being normal in the $p_0$-Zariski dense subgroup $\Delta_2$.
This contradicts  assumption (4).
\end{proof}

\begin{cor}
\label{cor:irreducibility of invariant random subgroups}
Let $\nu$ be a non-trivial discrete ergodic invariant random subgroup of $G$. If $\nu$-almost every subgroup intersects trivially  each proper semisimple factor of $G$ then $\nu$ is irreducible. Furthermore, any ergodic probability measure preserving $G$-space $(X,\overline{\nu})$ satisfying $(\mathrm{Stab}_G)_* \overline{\nu} = \nu$ is irreducible as well.
\end{cor}
\begin{proof}
We claim that $\nu$-almost every subgroup projects densely to any \emph{proper} semisimple factor $H$ of the standard semisimple group $G$. The claim will be established by induction on the number of types (i.e. either $\infty$ or prime numbers) involved in the factor $H$.
Write $G = H \times L$ for some suitable non-trivial factor $L$. Let $\nu_H$ denote  the pushforward of $\nu$ via the map $\Sub{G} \to \Sub{H}$ given by $\Lambda \mapsto \overline{\mathrm{pr}_H(\Lambda)}$ so that $\nu_H$ is an invariant random subgroup  of the group $H$. 

For the base of the induction, assume that the proper semisimple factor $H$ is   of a single  type $p$ (where $p$  is either $\infty$ or some prime number). By  Proposition \ref{prop:BDT for IRS} applied to the invariant random subgroup $\nu_H$, there is a pair of normal subgroups $N,M \le H$ with $N \cap M = \{e\}$ such that $\nu_H$-almost every subgroup projects densely to  $N$, discretely and $p$-Zariski-densely to $M$ and is contained in $N \times M$. Lemma \ref{lemma:types of intersection} applied with respect to the direct product decomposition $G =  (N \times L) \times M$  implies that the subgroup $M$ is trivial so that $N = H$. This means that $\nu_H = \delta_H$. The claim in the base of the induction is established.

For the induction step, assume that $H$ is a proper semisimple factor involving more than a single type, and that the claim has already been established with respect to all factors with fewer types. Let $p$ be any \emph{finite}  prime such that $H^{(p)}$ is non-trivial, and write $H = H^{(p)} \times R$ where $R = \prod_{q \neq p} H^{(q)}$ is the complement to $H^{(p)}$ in $H$. As the local field $\mathbb{Q}_p$ is non-Archimedean,    the standard semisimple group $H^{(p)}$ admits some compact open subgroup $O$. Consider the map 
$$\Sub{G} \to \Sub{R \times L}, \quad \Gamma \mapsto \mathrm{pr}_{R \times L}(\Gamma \cap \mathrm{pr}_{H^{(p)}}^{-1}(O) ) \quad \forall \Gamma \in \Sub{G}.$$ 
Let $\lambda_O$ be the resulting pushforward invariant random subgroup of the standard semisimple group $R \times L$. Note that $\lambda_O$ is non-trivial, since $\nu$-almost every subgroup projects densely to $H^{(p)}$ by the base case of the induction. Therefore, by the induction hypothesis (applied to the factor $R$ of the  group $R \times L$) we know that $\lambda_O$-almost every subgroup projects densely to the proper factor $R$. By repeating this argument with respect to a sequence $O_n$ of compact open subgroups of  $H^{(p)}$ which serves as a neighborhood basis at the identity,   we deduce that $\nu_H$-almost every subgroup \emph{contains} the factor $R$.
On the other hand, we know that $\nu_H$-almost every subgroup projects densely to the factor $H^{(p)}$, as was already observed above. From these two facts we conclude that $\nu_H = \delta_H$. The claim is established.

We know that $\nu$ projects densely to each proper semisimple factor of $G$.
Hence $\nu$ is irreducible by \cite[Corollary 7.3]
{fraczyk2023infinite}.
The additional clause in the statement concerning the irreducibility of $\overline{\nu}$ follows from the same \cite[Corollary 7.3]{fraczyk2023infinite}.
\end{proof}

\subsection*{Spectral gap}

We specialize our spectral gap result (Theorem \ref{theorem:getting spectral gap - general case}) for actions of standard semisimple groups.

\begin{theorem}[Spectral gap for actions of products --- semisimple group case]
\label{theorem:getting spectral gap - analytic groups}
Let $G$ be a standard semisimple group with $G = G_1 \times G_2$ where $G_1$ is standard semisimple and $G_2$ is standard \emph{simple} of type $p_0$. Let $X$ be a locally compact topological $G$-space endowed with a $G$-invariant  measure $m$, either finite or infinite. Assume that
\begin{itemize}
\item $L^2_0(X,m)^{G_2} = 0$,
\item  $m$-almost every point $x$ has $\mathrm{Stab}(x) \cap G_1 = \{e\}$ and
\item if $f_n \in L^2(X,m)$ is an asymptotically $G_1$-invariant sequence of unit vectors then every  accumulation point  $\nu \in \Prob(\Sub G)$ of the sequence of probability measures $\mathrm{Stab}_*(|f_n|^2 \cdot m)$ is such that $\nu$-almost every subgroup is  discrete,  not contained in the factor $G_2$  and admits $p_0$-Zariski-dense and not relatively compact projections to $G_2$.
\end{itemize}
Then the unitary Koopman $G$-representation $L^2_0(X,m)$ has spectral gap.
\end{theorem}
\begin{proof}
The idea is to deduce the spectral gap for the Koopman representation from our general spectral gap theorem for products, namely Theorem \ref{theorem:getting spectral gap - general case}. With that goal in mind, let us verify the two assumptions of that theorem. The first assumption that $L^2_0(X,m)^{G_2} = 0$ is maintained in the current statement as well.
As for the second assumption,
consider some asymptotically $G_1$-invariant sequence of unit vectors $f_n \in L_0^2(X,m)$. Denote $\nu_n = \mathrm{Stab}_*(|f_n|^2\cdot m) \in \mathrm{Prob}(\Sub{G})$ and let $\nu \in \mathrm{Prob}(\Sub{G})$ be any weak-$*$ accumulation point of the sequence of probability measures $\nu_n$. We are required to show that $\nu$-almost every subgroup has dense projections to the factor $G_2$.
Note that the probability measure $\nu$ is $G_1$-invariant.
By our assumptions, 
 $\nu$-almost every subgroup $\Gamma$ satisfies
\begin{enumerate}
    \item $\Gamma$ is discrete,
    \item $\Gamma$ is not contained in the factor $G_2$, and
    \item $\Gamma$ has $p_0$-Zariski-dense and not relatively compact projections to  $G_2$.
\end{enumerate}
Moreover, by applying part (1) of Corollary~\ref{cor:special cases of conjugate limit with Zariski dense intersection} to the subgroup $\Delta=\Stab_G(x)$ with respect to  a $m$-generic point $x\in X$ we get that
\begin{enumerate}
\setcounter{enumi}{3}
    \item if $\Gamma \cap G_2$ is $p_0$-Zariski-dense then $\Gamma \cap G_1$ is trivial.
\end{enumerate}
Applying Proposition~\ref{prop:getting dense projection to G2} we deduce that $\nu$-almost every subgroup has dense projections to the factor $G_2$. This concludes the reduction of the current proof to the statement of Theorem \ref{theorem:getting spectral gap - general case}.
\end{proof}

\begin{remark}
If the action of a standard semisimple group $G = G_1 \times G_2$ on a probability measure space $(X,m)$ is faithful, irreducible and  measure preserving then the stabilizer of $m$-almost every point has trivial intersection with the factors, so that this assumption in Theorem \ref{theorem:getting spectral gap - analytic groups} becomes  redundant.
\end{remark}

\begin{remark}
Provided $m$ is not supported on a singleton, the assumptions of Theorem \ref{theorem:getting spectral gap - analytic groups}  imply that both semisimple factors $G_1$ and $G_2$ are not trivial.
\end{remark}

\section{Confined and strongly confined subgroups}
\label{sec:confined and irreducibly}

In this section we study the notion of confined subgroups, which is a crucial ingredient in this work.

We consider this notion for discrete groups first.
Recall that a subgroup $\Lambda$ of a discrete group $\Gamma$ is  called {\it confined}  if there is a finite subset $F\subset \Gamma\setminus\{e\}$ such that the condition $\Lambda^\gamma\cap F\ne\emptyset$ holds true for every element $\gamma\in \Gamma$. 
Here is an equivalent definition for the negation of this notion. A subgroup $\Lambda$ of a discrete group $\Gamma$   is called {\it unconfined} (i.e. not confined) if the trivial subgroup of $\Gamma$ is a conjugate limit of $\Lambda$. Thus, being  unconfined and non-trivial is a vast strengthening of being non-normal.  
Indeed, non-trivial normal subgroups are precisely the fixed points for the conjugation action of $\Gamma$ on the space $\text{Sub}(\Gamma)\setminus\left \{\left\{e\right \} \right\}$, while unconfined subgroups are those with unbounded (namely, non-relatively compact) orbits. 

We shall require a generalization of this definition to the context of locally compact groups. 
 
\begin{defn}
\label{defn:confined subgroup}
 A closed subgroup $H$ of a locally compact second countable group $G$ is called \emph{confined} if the trivial subgroup $\{e\} \le G$ is not a conjugate limit of $H$. 
\end{defn}

Equivalently, a subgroup $H$ is confined in $G$ if and only if  there is a compact set $C\subset G\setminus\{e\}$ which intersects non-trivially every conjugate of $H$.
It follows immediately from the definition that any conjugate limit of a confined   subgroup is confined. 

\begin{eg}
Any non-trivial normal closed subgroup is confined. Any closed subgroup containing a confined subgroup is confined. In particular, a closed subgroup containing a non-trivial normal closed subgroup is confined.
\end{eg}

In the general context of locally compact groups, it is natural to consider the following variant of the definition.

\begin{defn}
Let $G$ be a locally compact second countable group. A closed subgroup $H$ of $G$ is {\it weakly confined} if there is a compact subset $C\subset G$ satisfying
$$
 (C\cap H^g) \setminus \{e\}\ne \emptyset
$$
for every element $g\in G$.    
\end{defn}

Confined subgroups are obviously weakly confined. Every non-discrete subgroup of $G$ is weakly confined. In the $p$-adic Lie group $\SL_2(\mathbb{Q}_p)$, the compact upper-triangular unipotent subgroup with coefficients in $\mathbb{Z}_p$ gives an example of a weakly confined subgroup which is unconfined. However, for groups with NSS (no small subgroups) property\footnote{A topological group has the \emph{NSS   property} (no small subgroups) it is admits an identity neighborhood containing no non-trivial closed subgroups.}  the two notions coincide \cite[Proposition 10.2]{gekhtman2023stationary}. In particular,   if $G$ is a real Lie group (possibly with infinitely many connected components) then every weakly confined subgroup is confined.

The property of being weakly confined generalizes non-trivial  normal subgroups and also lattices (just as the notion of invariant random subgroups generalizes those). 

\begin{lem}
\label{lema:lattices are confined}
    Let $G$ be a   locally compact second countable group. Any non-trivial lattice in $G$ is weakly confined.
\end{lem}

Note that the trivial subgroup is a lattice in $G$ if and only if the group $G$ is compact.

\begin{proof}[Proof of Lemma \ref{lema:lattices are confined}]
Let $\Gamma$ be a non-trivial lattice in $G$.
In view of \cite[Theorem 1.12]{raghunathan1972discrete} or \cite[Lemma 3.1]{GelanderPCMI}, there is a compact subset $K\subset G$ such that for every element $g\notin K$ the intersection $\Gamma^g\cap K$ contains a non-trivial element.
Let $\gamma\in \Gamma$ be any non-trivial element. Then the compact set $C = K\cup \gamma^K$ intersects   every conjugate of the lattice $\Gamma$ in a non-trivial element.
\end{proof}

Thus, lattices in non-compact real Lie groups as well as finite-index subgroups of infinite discrete groups are confined. 

Following \cite{KM-IRS} we will say that the locally compact group $G$ has the \emph{NDSS property} (i.e. no discrete small subgroups) if there is an identity neighborhood $U\subset G$ which contains no non-trivial finite subgroups. Obviously, if $G$ has NDSS then a discrete subgroup of $G$ is confined if and only if it is weakly confined. 

\begin{cor}
Let $G$ be a locally compact second countable group with NDSS. Then all lattices in $G$ are confined.    
\end{cor}

The class of NDSS groups contains all real and $p$-adic Lie groups and  is closed under products. Therefore we obtain the following.

\begin{lem}
Lattices in standard semisimple groups are confined.   
\end{lem}

In center-free semisimple real Lie groups there is a geometric criterion for a discrete subgroup to be confined. 
Let $G$ be such a  Lie group with associated symmetric space $X=G/K$, where $K$ is a maximal compact subgroup of $G$.  A discrete subgroup $\Lambda \le G$ is confined if and only if there is some $R > 0$ such that the injectivity radius of the orbifold $M_\Lambda = \Lambda \backslash X$ is upper bounded by $R$ at all points of $M_\Lambda$.
Such orbifolds  are called {\it uniformly slim} in \cite{fraczyk2023infinite}. 

Generally speaking,   being a confined subgroup is not a transitive notion.  The following example shows that even a normal subgroup $\Lambda$ of a lattice $ \Gamma$ in a simple Lie group $G$ may not in itself be confined in $G$. Note that this normal subgroup $\Lambda$ is even co-amenable in $\Gamma$ and therefore also in $G$.

\begin{eg}[Being confined is not transitive] \label{ex:confinedinGamma}
Consider the simple Lie group $G=\PSL(2,\real)$ and the lattice $\PSL(2,\mathbb{Z}) \le G$. Recall that the lattice $\PSL(2,\mathbb{Z})$ is isomorphic to the free product 
\[\PSL(2,\mathbb{Z}) \cong \ints/2\ints *  \ints/3\ints.\] 
Let $\Gamma$ be given by the following exact sequence 
\[ 1 \to \Gamma \to \PSL(2,\ints) \to  \ints/2 \ints \times \ints/3\ints  \to 1.\]
We have $\left[\PSL(2,\ints) : \Gamma\right] = 6$ so that $\Gamma$ is a lattice in $G$. 
The group $\Gamma $  is isomorphic to the free group $F_2$.
We may  view $\Gamma$ as the fundamental group of a three-holed   sphere. The holes are cusps in the corresponding finite-volume hyperbolic metric. 
Let $a,b$ represent the primitive loops around two of the cusps. Then $\Gamma=\langle a,b\rangle$ and $ab$ is a loop around the third cusp.
Every unipotent element of   $\Gamma$ is conjugate to either a power of $a$, a power of $b$ or a power of $ab$.

Let $\Lambda = \left[\Gamma,\Gamma\right]$ be the commutator group of $\Gamma$.
Then $\Lambda$ is normal in $\Gamma$ and co-amenable in $G$. The subgroup $\Lambda$  contains no non-trivial unipotent elements. Indeed,    the images of $a,b$ precisely generate the abelianization $\Gamma/\Lambda\cong \ints^2$.

In this situation, if a sequence of elements  $g_n\in G$ is such that $g_n \Gamma \to \infty$ in the quotient $G/ \Gamma$ then the sequence of conjugates $g_n \Lambda g_n^{-1}$ tends to the trivial subgroup in the Chabauty topology on $\sub(G)$. Equivalently, if a sequence of points  $x_n$ tends to infinity in one of the cusps in the hyperbolic surface $\Gamma\backslash G/K$ then the injectivity radius at any lift $\widetilde{x}_n$ of the point $x_n$ in $\Lambda\backslash G/K$ tends to infinity. This is because any   loop through $x_n$ in $\Gamma\backslash G/K$ represented by a non-unipotent element must wind around another hole, hence pass through the thick part.
\end{eg}

\subsection*{Confined subgroups in ICC groups}

Recall that a group $\Gamma$ is called \emph{ICC (i.e. infinite conjugacy classes)} if the conjugacy class of every non-trivial element of $\Gamma$ is infinite. We consider the properties of confined subgroups in ICC groups.

\begin{lemma}
\label{lemma:no finite confined subgroups in ICC}
Let $\Gamma$ be a discrete ICC group. If $F$ is a finite subgroup of $\Gamma$ then $F$ is not confined in $\Gamma$.
\end{lemma}
\begin{proof}
Let $F$ be a finite subgroup of $\Gamma$. Assume that $|F| > 1$, for otherwise there is nothing to prove. So the subgroup $F$ admits some non-trivial element $h \in F$. Since the group $\Gamma$ is ICC, there is a sequence of elements  $\gamma_i \in \Gamma$ such that the conjugates $h^{\gamma_i}$ are pairwise distinct. Up to passing to a subsequence, we may assume that the limit $F' = \lim_{i} F^{\gamma_i}$ exists in the Chabauty topology. Note that the conjuguate limit $F'$ satisfies $|F'|<|F|$. We conclude that the trivial subgroup is a conjugate limit of the given subgroup $F$ arguing by induction on $|F|$. Namely $F$ is not confined.
\end{proof}


\begin{lemma}
\label{lemma:finite index of ICC}
Let $\Gamma$ be a discrete group without finite confined subgroups. Then the notion of being confined for subgroups of $\Gamma$ is a commensurability invariant.
\end{lemma}

\begin{proof}
It is enough to check that the property is preserved by passing to  finite-index subgroups.
Let $\Delta$ be a subgroup of $\Gamma$ and $\Delta_1\le\Delta$ be a subgroup of finite index.
Assume towards contradiction that $\Delta$ is confined in $\Gamma$ but $\Delta_1$ is not. Let $\gamma_i \in \Gamma$ be a sequence of elements such that the   subgroups $ \Delta_1^{\gamma_i}$ converge to the trivial subgroup in the Chabauty topology. Up to passing to a subsequence, we may assume that the Chabauty limit $\Lambda = \lim_i \Delta^{\gamma_i}$ exists. We claim that   $\Lambda$ is finite. Consider a pair of distinct non-trivial elements $\delta_1,\delta_2 \in \Lambda$. Then $\delta_1,\delta_2 \in (\Delta \setminus \Delta_1)^{\gamma_i} $ for all $i$ sufficiently large. Note that $\delta_1$ and $\delta_2$ belong to different cosets in $\Delta^{\gamma_i}/\Delta_1^{\gamma_i}$ for all $i$ sufficiently large, for otherwise $\delta_1^{-1} \delta_2 \in  \Delta_1^{\gamma_i}$ for arbitrarily large $i$, which is impossible. Therefore   $|\Lambda| \le \left[\Delta:\Delta_1\right] < \infty$. On the other hand, note that the subgroup $\Lambda$ is confined, being a conjugate limit of the confined subgroup $\Delta$. A contradiction.
\end{proof}

This lemma applies in particular to torsion-free groups as well as to ICC groups in view of Lemma \ref{lemma:no finite confined subgroups in ICC}. We obtain the following.

\begin{cor}
\label{cor:in ICC confined in commensurability invariant}
Let $\Gamma$ be a discrete group. 
Assume that $\Gamma$ is either torsion-free or ICC.
Then the notion of being confined for subgroups of $\Gamma$ is invariant under commensurability.
\end{cor}
 



 

\subsection*{Confined actions}

Let $G$ be a locally compact second countable group. For every Borel $G$-space $X$ there is a Borel map $\mathrm{Stab}_G : X \to \mathrm{Sub}(G)$ given by $x \mapsto \mathrm{Stab}_G(x)$.

\begin{defn} 
\label{defn:confined actions}
We introduce several notions related to confined actions.

\begin{enumerate}[label=(\alph*)]
\item A uniformly recurrent subgroup $X$, or more generally a closed $G$-invariant subset $X$ of $\Sub{G}$, is called \emph{confined} if $\{e\} \notin X$.
\item An invariant random subgroup $\nu \in \IRS(G) $ is called \emph{confined} if $\mathrm{supp}(\nu)$ is confined.
\item A topological $G$-space $X$ has \emph{confined stabilizers} if the closed $G$-invariant subset $\overline{\mathrm{Stab}_G(X)} \subset \Sub{G}$ is confined.
\item A probability measure preserving Borel $G$-space $(X,\mu)$ has \emph{confined stabilizers} if the invariant random subgroup $(\mathrm{Stab}_G)_* \mu$ is confined.
\end{enumerate}
\end{defn}

We caution the reader that, say, the non-confined invariant random subgroup $\delta_{\{e\}}$ has confined stabilizers as  a probability measure preserving space.

\begin{example}
Let $H$ be a closed subgroup of $G$. The following statements are equivalent:
\begin{enumerate}
    \item The closed subgroup $H \le G$ is confined in the sense of Definition \ref{defn:confined subgroup}.
    \item The orbit closure $\overline{H^G}$ is a confined subset of $\Sub{G}$ in the sense of (a) in  Definition \ref{defn:confined actions}.
    \item The topological $G$-space $G/H$ has confined stabilizers in the sense of (c) in  Definition \ref{defn:confined actions}.
\end{enumerate}
\end{example}

 \begin{lemma}
\label{lemma:compact confined model}
Assume that the group $G$ is discrete. Let $\nu$ be a confined invariant random subgroup of $G$. Then there is a    compact $G$-space $Y$ with confined stabilizers admitting a $G$-invariant probability measure $\mu$ such that $(\mathrm{Stab}_G)_* \mu = \nu$.
\end{lemma}
\begin{proof}
According to \cite[Proposition 13]{abert2014kesten} there is a Borel probability measure preserving $G$-space $(Z,\eta)$ satisfying $(\mathrm{Stab}_G)_* \eta = \nu $. By Varadarajan's compact model theorem, there is a compact $G$-space $(Y,\mu)$ with $\mathrm{supp}(\mu) = Y$ and such that $(Z,\eta)$ and $(Y,\mu)$ are isomorphic as Borel $G$-spaces. In particular $(\mathrm{Stab}_G)_* \mu = \nu$ so that $(Y,\mu)$ has confined stabilizers as a probability measure preserving space.

To conclude the proof it remains to show that the compact $G$-space $Y$ has confined stabilizers as a topological space. Note that $\mu$-almost every point $y \in Y$ has $\mathrm{Stab}_G(y) \in \mathrm{supp}(\nu)$. 
As the group $G$ is discrete, the map $\mathrm{Stab}_G : Y \to \mathrm{Sub}(G)$ is upper semi-continuous, in the sense that any point $x \in X$ has a neighborhood $x \in O \subset X$ such that any point $y \in O$ satisfies $\mathrm{Stab}_G(y) \le \mathrm{Stab}_G(x)$. Put together, these two facts imply that $\overline{\mathrm{Stab}_G(Y)}$ is confined, as required.
\end{proof}

\subsection*{Confined subgroups of arithmetic lattices}

We study confined subgroups of irreducible lattices and their actions factoring through rank-one simple factors.
\begin{lemma}
\label{lemma:confined subgroups of arithemtic lattices}
Let $G$ be a standard semisimple group of higher rank and $\Gamma$ be an irreducible lattice in $G$. Let $F$ be a rank one simple factor of $G$ of type $p$, where $p$ is either $\infty$ or a prime number. Let $\Delta$ be a confined subgroup of $\Gamma$. Then the projection of $\Delta$ to the factor $F$ is $p$-Zariski-dense and not relatively compact.
\end{lemma}

\begin{proof}
The irreducible lattice $\Gamma$ is arithmetic by  Margulis' arithmeticity (Theorem~\ref{thm:arithmeticity}).
We will use  the notations introduced in Examples~\ref{ex:H} and \ref{ex:Gamma_S}.
Namely  $K$ is a number field, ${\bf H}$ is an adjoint connected absolutely simple $K$-group and $S$ is a finite set of isotropic places on $K$ containing all Archimedean isotropic ones. We identify the standard semisimple group $G$ with the group $H_S$ and the rank one simple factor $F$ with the factor   $H_s$ for some place $s \in S$.  Let $k_s$ be the local field associated to the place $s$.

The lattice $\Gamma$ is ICC according to  Lemma~\ref{lemma:lattices in simple analytic groups are ICC}.
Therefore   Corollary \ref{cor:in ICC confined in commensurability invariant} applies, and we may   replace $\Gamma$ and $\Delta$ by finite-index subgroups preserving the assumptions of the lemma.
Upon doing so and conjugating, we assume as we may that $\Gamma=\Gamma_S$ where $\Gamma_S$ was defined in Example \ref{ex:Gamma_S}. In particular, the lattice $\Gamma$ is contained in the group of rational points $\mathbf{H}(K)$.

Let $X$ be the rank-one symmetric space or Bruhat--Tits building associated to the standard simple group $F$, depending on whether $p$ is $\infty$ or a prime.  In both cases $X$ is a proper Gromov hyperbolic metric space.
Its Gromov boundary $\partial X$ can be identified with   $({\bf H}_s/{\bf P})(k_s)$ where $\bf P$ is a minimal parabolic $k_s$-subgroup of ${\bf H}_s$ \cite[Propoisiton 20.5]{borel2012linear}. The variety  ${\bf H}_s/{\bf P}$ is in fact defined over some finite  extension $L$ of the field $s(K)$ \cite[Corollary 18.8]{borel2012linear}.
We consider $\Gamma$ as a subgroup of $\mathbf{H}(L)$.

Let $\mu$ be any symmetric probability measure on the lattice $\Gamma$  whose support generates $\Gamma$. Let $\nu$ be any \mbox{$\mu$-stationary} random subgroup supported on the conjugate closure $\overline{\Delta^\Gamma}$. The assumption that the subgroup $\Delta$ is confined ensures that $\nu$-almost every subgroup of $\Gamma$ is not trivial.


Consider the action of the lattice  $\Gamma $ on the Gromov boundary   $({\bf H}_s/{\bf P})(k_s)$.
We conclude that the fixed point set of  every element in $\Gamma$ is defined over $L$, hence so is the fixed point set of every subgroup of $\Gamma$.
We know by \cite[Proposotion 4.7]{gekhtman2023stationary} that $\nu$-almost every subgroup fixes at most a single point of the boundary $\partial X$. Every such fixed point must belong to the countable set $({\bf H}_s/{\bf P})(L)$ by the preceding remarks. However, any $\nu$-stationary probability measure on a countable set is finitely supported  and $\Gamma$-invariant  \cite[Lemma 8.3]{benoist2011mesures}. We conclude that $\nu$-almost every subgroup fixes no point of the boundary $\partial X$. Therefore its projection to $F$ is acting on $X$ without   proper closed convex invariant subsets \cite[Corollary 7.6]{gekhtman2023stationary}. This implies that $\nu$-almost every subgroup projects 
$p$-Zariski-densely to $F$ \cite[Proposition 2.8]{caprace2009isometry}.
The fact that the projection of the group $\Delta$ itself (rather than its  conjugate limit) to the factor $F$ is $p$-Zariski-dense  follows from Lemma \ref{lem:Z-dense}.

Similarly, we know that $\nu$-almost every subgroup acts on $X$ with unbounded orbits \cite[Proposition 4.8 or 7.8]{gekhtman2023stationary}. By this fact and by the previous paragraph, this action cannot be bounded or horocyclic (in the sense outlined e.g.  in \cite[\S3]{gekhtman2023stationary}). Hence the action   admits hyperbolic elements. Being hyperbolic is an open condition in the group $F$ \cite[Proposition 3.1]{gekhtman2019critical}. We conclude that the group $\Delta$ itself admits an element whose projection to $F$ is hyperbolic. As such, the projection of the subgroup $\Delta$ to the factor $F$ is not relatively compact.
\end{proof}

\subsection*{Strongly confined subgroups}

We require a more refined notion than just being confined, which takes into account  degeneration of conjugate limits into proper factors. 

\begin{defn}
\label{defn:strongly confined}
A closed subgroup $H$ of a locally compact second countable group $G$ is \emph{strongly confined} if no conjugate limit of $H$ is  contained in a proper normal subgroup.
\end{defn}

Certainly, a strongly confined subgroup is confined. Moreover, any conjugate limit of a strongly confined subgroup is still strongly confined.


The following observation is needed to obtain  the converse (easier) direction to one of our main results.

\begin{lemma}
\label{lemma:confined subgroup of an irreducible lattice is strongly confined}
Let $G$ be a standard semisimple group and  $ \Gamma$ be an irreducible lattice in $G$. Then any subgroup of $\Gamma$ which is confined regarded as a subgroup of $G$ is strongly confined in $G$.
\end{lemma}
 
\begin{proof}
We assume as we may that the group $G$ is semisimple but not simple, for otherwise the two notions of confined and strongly confined are  equivalent. In particular $G$ is of higher rank and  the lattice $\Gamma$ is arithmetic (see Theorem~\ref{thm:arithmeticity}). 

We use the notation introduced in Examples~\ref{ex:H} and \ref{ex:Gamma_S}
and  identify $G$ with the group $H_S = \prod_{i=1}^m H_{s_i}$ where $S = \{s_1,\ldots,s_m\}$ is a set of places. The lattice $\Gamma$ is ICC according to Lemma \ref{lemma:lattices in simple analytic groups are ICC}.
Up to conjugating the lattice $\Gamma$ and   using Corollary \ref{cor:in ICC confined in commensurability invariant} to replace $\Gamma$ by a finite-index subgroup if necessary, we assume as we may that $\Gamma=\Gamma_S$. In particular $\Gamma$ contained in $\mathbf{H}(K)$ for some number field $K$.

Let   $\Lambda \le \Gamma$ be a subgroup. Assume towards contradiction that $\Lambda$ is confined but not strongly confined regarded as a subgroup of $G$. Upon reordering $S$, we get that there is a sequence of elements $g_n \in G$ such that $\Lambda^{g_n} \to \Delta$ in the Chabauty topology with   $\Delta \le L$ and $L = \prod_{i=1}^{m-1}H_{s_i}$. Note that $\Delta \neq \{e\}  $ since $\Lambda$ is confined. This means that for each element $h \in \Delta$ there is a sequence of elements $\gamma_n = (\gamma_{1,n},\ldots,\gamma_{m,n}) \in \Delta$ where $\gamma_{i,n} \in H_{s_i}$ so that   $(\gamma_{1,n},\ldots,\gamma_{m-1,n})^{g_n} \to h$ and at the same time $\gamma_{m,n}^{g_n} \to e$. The coefficients of the characteristic polynomials of  the transformations $\mathrm{Ad}(\gamma_{1,n}),\ldots, \mathrm{Ad}(\gamma_{m,n})$  are all uniformly bounded and  the characteristic polynomials of $\mathrm{Ad}(\gamma_{m,n})$ tend to the polynomial $p(x)=(x-1)^{\dim_K {\bf H}}$ as $n$ tends to infinity. As the algebraic $S$-integers form a lattice in the   product of local fields $\prod_{s \in S} k_s$,  the set of all possible coefficients of such  characteristic polynomials is finite. Hence
the projection of the element $\gamma_n$ to each $p$-component must be unipotent for all $n$ sufficiently large. It follows that every element belonging to the projection of the subgroup $\Delta$ to each $p$-component is unipotent. 
This implies that the $p$-Zariski closure $U_p$ of the projection of $\Delta$ to each $p$-component is a connected unipotent subgroup.  

In each $p$-component of the group $L$, there is some horospherical subgroup $V_p$ containing the unipotent subgroup $U_p$ \cite[Corollary 3.7]{borel1971elements}. Let $V$ be the direct product of these horospherical subgroups. Let $a(t)$ be a suitable one-parameter subgroup of $L$ which expands the subgroup $V$, namely $v^{a(t)} \xrightarrow{t\to\infty} \infty$ holds true for every element $v \in V$. 

For a given radius $ R > 0$ let $B_e(R)$ denote the ball at the identity of radius $R$ in the group $G$ with respect to some fixed proper and continuous metric.
As the subgroup $\Delta$ is discrete,  for each $i \in \mathbb{N}$ there is some sufficiently large $ t_i >0 $ such that 
$$((\Delta \cap B_e(i))\setminus \{e\})^{a_1(t_i)} \subset G \setminus B_e(2 i)$$ 
as well as $d(x^{a_1(t_i)},e) \ge 2d(x,e)$ for all $x \in V$.
Let $R_i > i$ be a sufficiently large radius  such that 
$$(G \setminus B_e(R_i))^{a_1(t_i)} \subset G \setminus B_e( i)$$ 
for each $i$. Let $j = j(i)$ be a sufficiently large index, such that every element of the intersection $\Lambda^{g_{j(i)}} \cap B_e(R_i)$ is \enquote{sufficiently close} to some element of $\Delta \cap B_e(R_i) $ for each $i$. More precisely, we require as we may for each $i$ that 
\[ ((\Lambda^{g_{j(i)}} \cap B_e(R_i))\setminus\{e\})^{a_1(t_i)}   \subset G \setminus B_e(i).\]
It follows that the sequence of conjugates $\Lambda_i 
 =\Lambda^{g_{j(i)} a_1(t_i)}$ converges to the trivial subgroup in the Chabauty topology on $\Sub{G}$. This is a contradiction to the assumption that $\Lambda$ is confined when regarded as a subgroup of $G$.
\end{proof}





 Confined discrete subgroups of rank one simple Lie groups are Zariski dense, see \cite[Lemma 9.14]{fraczyk2023infinite}. We show that projections of strongly confined discrete subgroups to simple rank one factors  are also Zariski dense, under certain conditions.

\begin{lemma}
\label{lemma:discrete irr confined has Zariski dense projections}
Let $G$ be a standard semisimple group of type $p$, where $p$ is either $\infty$ or a prime number. Let $H$ be a simple factor of $G$ with $\mathrm{rank}(H) = 1$.  Let $\Delta$ be a strongly confined subgroup of $G$. Then the projection of $\Delta$ to the factor $H$ is $p$-Zariski-dense and not relatively compact.
\end{lemma}

\begin{proof}
Write $G = H \times H'$ where $H'$ is a suitable standard semisimple  subgroup. Let $\mu$ and $\mu'$ be a pair of   probability measures on the groups $H$ and $H'$ respectively, specifically the particular ones considered in \cite{gelander2022effective}. Denote $\mu_0 = \mu \otimes \mu'$ and let $\nu$ be any $\mu_0$-stationary limit  of the subgroup $\Delta$ (in the sense of Definition \ref{def:mu-stationary limit}). In particular $\nu$ is a $\mu$-stationary random subgroup supported on the conjugate closure $\overline{\Delta^G}$.   
In the real case, 
we know that $\nu$-almost every subgroup of $G$ is discrete by  \cite[Theorem 1.6]{fraczyk2023infinite}. In the $p$-adic case, $\nu$-almost every subgroup is certainly discrete, as $G$ has a compact open subgroup containing no no-trivial discrete subgroups \cite[Theorem 1 on p. 124]{serre2009lie}. In any case $\nu$-almost every subgroup is not contained in the factor $H'$ by the strongly confined assumption. The fact that $\nu$-almost every subgroup fixes no point in its action on the Gromov boundary of the rank one symmetric space or Bruhat--Tits building associated to the rank one group $H$  follows immediately from \cite[Theorem 6.1]{gekhtman2023stationary}. From this point onward, we conclude the proof exactly as in  Lemma \ref{lemma:confined subgroups of arithemtic lattices}.
\end{proof}

The above two Lemmas \ref{lemma:confined subgroups of arithemtic lattices} and \ref{lemma:discrete irr confined has Zariski dense projections} are very much closely related. However,    a priori a confined subgroup of the lattice $\Lambda$ may not be confined regarded as a subgroup of the enveloping group $G$, as was the case in Example \ref{ex:confinedinGamma}.

\section{Confined subgroups of irreducible lattices}
\label{sec:confined subgroups of lattices}

The current section is devoted to the proof of Theorem \ref{thm: lattices, general} saying that any confined subgroup of an irreducible lattice of a higher rank standard semisimple group has finite index. This is the generalization of 
Theorem \ref{thm intro: lattices} as well as Corollary \ref{cor:lattices URS statement } of the introduction to the setting of standard semisimple    groups over local fields (rather than semisimple Lie groups).

If $\Gamma$ is a cocompact lattice in the locally compact group $G$ then every confined subgroup of $\Gamma$ is confined in $G$.
However, this need not be the case if $\Gamma$ is a non-cocompact lattice, see Example~\ref{ex:confinedinGamma}.
The next two lemmas are here to remedy this failure. They are not needed in the cocompact case.



For a Borel subset $M \subset \Sub{G}$, we denote
\[ M^G=\{ H^g \mid H\in M,~g\in G\} \]
and 
\[ \Prob(M)=\{\nu \in \Prob(\Sub G) \: : \: \nu(M)=1\}.\]

\begin{lemma}
\label{lemma:the good set C}
Let $\Gamma$ be a lattice in the locally compact group $G$. Let $X$ be a topological $\Gamma$-space and  $Y = G \times_\Gamma X$ be the induced topological $G$-space\footnote{For the notion of induced actions we refer e.g. to \cite[p. 75]{zimmer2013ergodic}.} with projection $\pi : Y \to G/\Gamma$. Denote
\[\mathcal{C}_\Gamma^G(X) = \{ \mu \in \mathrm{Prob}(Y) \: : \: \pi_* \mu = m_{G/\Gamma} \}.\]
 Then 
\[ \overline{(\Stab_G)_*(\mathcal{C}_\Gamma^G(X) )}\subset \Prob( \overline{\mathrm{Stab}_\Gamma(X)} ^ G).\]
\end{lemma}

\begin{proof}
We consider $\Sub{\Gamma}$ as a closed subset of $\Sub{G}$.
Denote $S = \overline{\mathrm{Stab}_\Gamma(X)}$ so that $S \subset \Sub{\Gamma}$ is a $\Gamma$-invariant closed subset. 
For each point $g\Gamma \in G/\Gamma$ consider the space of probability measures
\[ A_{g\Gamma} = \mathrm{Prob}\left( S^g \right) \]
regarded as a compact convex subset of $\mathrm{Prob}(\Sub{G})$. Consider the compact convex space 
\[Q  = F(G/\Gamma, \{A_{g\Gamma}\})\]
consisting of all $m_{G/\Gamma}$-measurable functions $f : G/\Gamma \to \mathrm{Prob}(\Sub{G})$ satisfying $m_{G/\Gamma}$-almost surely $f(g\Gamma) \in A_{g\Gamma}$, see \cite[p. 78]{zimmer2013ergodic}. The barycenter map with respect to the normalized Haar   measure $m_{G/ \Gamma}$ sets up a continuous affine map \[\mathrm{bar} : Q \to \mathrm{Prob}(S^G) \subset \mathrm{Prob}(\Sub{G}).\]

Note that the subset  $\mathrm{bar}(Q)$ is closed. Therefore, to conclude the proof it will suffice to construct an affine map $\varphi : \mathcal{C}_\Gamma^G(X) \to Q$ such that $\mathrm{Stab}_* = \mathrm{bar}\circ \varphi$ on the space $\mathcal{C}_\Gamma^G(X)$, which is what we proceed to do.

Given a probability measure $\mu \in \mathcal{C}_\Gamma^G(X) \subset \mathrm{Prob}(Y)$, consider its disintegration $d_\mu$ over the projection $\pi $ to the probability space $(G/\Gamma,m_{G/\Gamma})$, given by a $m_{G/\Gamma}$-measurable map
\[ d_\mu : G/\Gamma \to \mathrm{Prob}(Y)\]
such that $m_{G/\Gamma}$-almost surely $d_\mu(g\Gamma)$ gives full measure to the fiber $\pi^{-1}(g\Gamma)$. In particular $\mathrm{Stab}_* d_\mu(g\Gamma) \subset A_{g\Gamma}$ almost surely. The affine map
\[ \varphi : \mathcal{C}_\Gamma^G(X) \to Q, \quad \varphi(\mu) (g\Gamma) =\mathrm{Stab}_* d_\mu(g\Gamma) \quad \forall \mu \in \mathcal{C}_\Gamma^G(X), ~\forall g\Gamma \in G/\Gamma \]
is as required.
\end{proof}

\begin{lemma} \label{lem:understanding limits of confined in a lattice}
Let $G$ be a standard semisimple group and $\Gamma<G$ an irreducible lattice.
Let $X$ be a topological $\Gamma$-space and  $Y$ be the induced topological $G$-space equipped with the projection $\pi : Y \to G/\Gamma$.  
Fix a $G$-invariant probability measure $\mu \in \mathrm{Prob}(Y)$ such that $\pi_* \mu = m_{G/\Gamma}$. 
 Let  $f_i\in L^2(Y,\mu)$ be an asymptotically $H$-invariant sequence of unit vectors for some non-trivial normal subgroup $H \lhd G$. 
Then any accumulation point   of the sequence of probability measures $\Stab_*(|f_i|^2  \cdot  \mu)$ in the space $\Prob(\Sub{G})$ belongs to  $\mathrm{Prob}( \overline{\mathrm{Stab}_\Gamma(X)}^G)$.
\end{lemma}

\begin{proof}
Consider the natural linear map
$ T  : L^1(Y,\mu) \to L^1(G/\Gamma)$ corresponding to the projection $\pi$.
It has operator norm $\|T\|_\textrm{op} = 1$. Furthermore $\|Tf\|_1 = \|f\|_1$ for non-negative functions (i.e. provided $f \ge 0$).

We consider the sequence of elements $|f_i|^2 \in L^1(Y,\mu)$.
 By Lemma~\ref{lem:mazur}  this is an asymptotically $H$-invariant sequence of non-negative unit vectors in $L^1(Y,\mu)$. 
We consider the probability measures $\nu_i=|f_i|^2  \cdot \mu \in \Prob(Y)$  and their push-forward to $\sub(G)$ under the map $\Stab:Y \to \sub(G)$.
Up to passing to a subsequence, we assume that the  probability measures $\mathrm{Stab}_* \nu_i$ converge to some probability measure $\zeta \in \mathrm{Prob}(\Sub{G})$. The asymptotic $H$-invariance of the sequence $|f_i|^2$ implies that the measure $\zeta$ is $H$-invariant.

Consider the sequence of functions $g_i \in L^1(G/\Gamma)$ given by $g_i = T (|f_i|^2)$. The sequence $g_i$ is asymptotically $H$-invariant, as $T$ is contracting. By  Clozel's theorem (essentially, see Theorem \ref{thm:Clozel} for details), the group $H$ has spectral gap in its representation on  $L^2_0(G/\Gamma)$. It follows by Lemma~\ref{lem:L1SG} that $\|g_i  - 1_{G/\Gamma} \|_1 \to 0$. 

Our strategy is to construct a sequence of probability measures $\eta_i$ all satisfying $\pi_* \eta_i = m_{G/\Gamma}$, such that $\mathrm{Stab}_* \nu_i$ and $\mathrm{Stab}_* \eta_i$ have the same asymptotic behaviour. Here are the details. Look at   the positive and negative pairs of the difference $1-g_i$. Namely
\[ 1-g_i = h_i^+ -  h_i^- \] 
where $h_i^+,h_i^- \in L^1(G/\Gamma)$, $h_i^+,h_i^-  \geq 0$ and $\|h_i\|_1=\|h_i^+\|_1+\|h_i^-\|_1$ . Let $l^+_i \in L^1(Y,\mu)$ be an arbitrary lift satisfying $T(l^+_i) = h^+_i$ and $l^+_i \ge 0$.
We note that $T(|f_i|^2+l^+_i)-h_i^-=1$. Thus $T(|f_i|^2+l^+_i) \ge h_i^-$.
Let $l_i^- \in L^1(Y,\mu)$ be an arbitrary lift with $T(l^-_i) = h^-_i$ and $|f_i|^2+l_i^+ \geq l_i^-\geq 0$.
We set $l_i=l_i^+-l_i^-$. Observe that 
\[ \|l_i\|_1 \leq \|l_i^+\|_1+\|l_i^-\|_1 =\|h_i^+\|_1+\|h_i^-\|_1=\|1-g_i\|_1 \xrightarrow{i \to \infty} 0.\]
We note that 
\[ |f_i|^2+l_i=|f_i|^2+l_i^+ - l_i^-\geq 0 \quad \mbox{and} \quad T(|f _i|^2+l_i)= T(|f_i|^2)+ 1- g_i = 1. 
\]
In particular, each $\eta_i = (|f_i|^2+l_i)\cdot m_{G/\Lambda}$ is a probability measure on $Y$. Since $T(|f_i|^2+l_i)= 1$,   the measures $\eta_i$
all satisfy $\pi_* \eta_i = m_{G/\Gamma}$. In particular $$\lim \mathrm{Stab}_* \eta_i \in \mathrm{Prob}(\overline{\mathrm{Stab}_\Gamma(X)}^G)$$ according to Lemma \ref{lemma:the good set C}.
On the other hand,   the fact that $\|l_i\|_1 \to 0$  implies 
\[ \zeta = \lim \Stab_* \nu_i = \lim\Stab_* \eta_i. \]
Since these two limits coincide, the desired conclusion follows.
\end{proof}

\begin{theorem}
\label{thm:confined ergodic action of lattice}
Let $G$ be a standard semisimple group with $\mathrm{rank}(G) \ge 2$ and $\Gamma$ an irreducible lattice in $G$. 
Then every ergodic confined invariant random subgroup of $\Gamma$ almost surely has finite index.
\end{theorem}
\begin{proof}
 Let $\nu$ be an ergodic confined invariant random subgroup of the group $\Gamma$. According to Lemma \ref{lemma:compact confined model} there is a compact confined $\Gamma$-space $X$ and a $\Gamma$-invariant probability measure $\mu \in \mathrm{Prob}(X)$ such that $(\mathrm{Stab}_\Gamma)_* \mu = \nu$.
Let $(Y,\overline{\mu})$ be the  probability measure preserving topological $G$-space induced from the $\Gamma$-space $(X,\mu)$. The invariant random subgroup $\overline{\nu} = (\mathrm{Stab}_G)_* \overline{\mu}$ coincides with the invariant random subgroup of the group $G$ 
induced from $\nu$ as defined in \S\ref{sec:prelim}.  Note that the  $G$-space $(Y,\overline{\mu})$ is ergodic \cite[Remark (1) on p. 75]{zimmer2013ergodic}. Hence the invariant random subgroup $\overline{\nu}$ is ergodic as well. Therefore the $G$-space $(Y,\overline{\mu})$ (as well as the invariant random subgroup $\overline{\nu}$) is irreducible by Corollary \ref{cor:irreducibility of invariant random subgroups}.  

We now claim that the unitary $G$-representation $L^2_0(Y,\overline{\mu})$ has spectral gap.
If  the group $G$ has Kazhdan's property (T) then
the statement  is immediate. Assume from now on that $G$ has no Kazhdan's property (T). As such, the group $G$ splits as a non-trivial direct product of standard semisimple groups $G = G_1 \times G_2$ where $G_1$ is a standard semisimple group and $G_2$ is a standard simple group of type $p_0$ and with $\mathrm
{rank}(G_2)=1$.

We would like to apply our Theorem \ref{theorem:getting spectral gap - analytic groups}
establishing spectral gap for actions of products with respect to the unitary representation $L^2_0(Y,\mu)$. The assumptions of that theorem are verified as follows:
\begin{itemize}
    \item The fact that $L^2_0(Y,\overline{\mu})^{G_2} = 0$ follows from the irreducibility of $\overline{\mu}$.
    \item The stabilizer of $\overline{\mu}$-almost every point is distributed according to $\overline{\nu}$, and as such, is conjugate to a subgroup of the irreducible lattice $\Gamma$ and  has trivial intersections with every proper normal subgroup of $G$.
    \item Let $f_n \in L^2(Y,\overline{\mu})$ be any asymptotically $G_1$-invariant sequence of vectors. Consider any weak-$*$ accumulation point $\zeta \in \Prob(\Sub{G})$ of the sequence of probability measures $\mathrm{Stab}_*(|f_n|^2 \cdot \overline{\mu}) $. Then according to Lemma \ref{lem:understanding limits of confined in a lattice}
     $\zeta \in \mathrm{Prob}(\overline{\mathrm{Stab}_\Gamma(X)}^G)$.
    This means that $\zeta$-almost every subgroup is discrete and not contained in any proper normal factor. Further, this implies that $\zeta$-almost every  subgroup is conjugated in $G$ to a confined subgroup of $\Gamma$. Thus by Lemma \ref{lemma:confined subgroups of arithemtic lattices} we see that $\zeta$-almost every subgroup has $p$-Zariski-dense and not relatively compact projections to $G_2$. 
\end{itemize}

Having verified all assumptions of Theorem \ref{theorem:getting spectral gap - analytic groups} we conclude that $L^2_0(Y,\overline{\mu})$ has spectral gap. By relying on the methods of Stuck--Zimmer \cite{SZ} and Bader--Shalom \cite{bader2006factor}, and making use of the product structure of the group $G$, this implies that the $G$-space $(Y,\overline{\mu})$ is essentially transitive, see e.g.  \cite[Proposition 7.6]{creutz2017stabilizers} or \cite[Theorem 3]{levit2020benjamini} for details.  
Therefore $\overline{\nu}$-almost every subgroup is a lattice in $G$. We conclude that $\nu$-almost every subgroup has finite index in $\Gamma$, as required.
\end{proof}

\begin{theorem}
\label{thm: lattices, general}
Let $G$ be a standard semisimple group with $\mathrm{rank}(G) \ge 2$ and $\Gamma$ an irreducible lattice in $G$. Then any confined subgroup of $\Gamma$ has  finite index. Further,  every non-trivial uniformly recurrent subgroup $X$ of $\Gamma$ is the set of conjugates of some finite-index subgroup of $\Gamma$.
\end{theorem}

\begin{proof}
Recall that the lattice $\Gamma$ is an ICC group by Lemma~\ref{lemma:lattices in simple analytic groups are ICC}. In particular, we may   replace $\Gamma$ by any of its finite-index subgroups without a loss of generality, see  Corollary \ref{cor:in ICC confined in commensurability invariant}. Specifically, 
using Margulis' arithmeticity (Theorem~\ref{thm:arithmeticity})
and the notation introduced in Examples~\ref{ex:H} and \ref{ex:Gamma_S},
we thus assume that the lattice $\Gamma$ is $S$-arithmetic in the sense that $\Gamma=\Gamma_S$. The significance of this fact for our purpose is that the arithmetic lattice $\Gamma$ is \emph{charmenable} by \cite[Theorem B]{BBH}. Further,  
note that $\Gamma$ has a trivial amenable radical.

Let $\Lambda$ be any confined subgroup of the lattice $\Gamma$. Consider any uniformly recurrent subgroup $X$ contained in the $\Gamma$-orbit closure $ \overline{\Lambda^\Gamma}$. Note that $X$ is non-trivial by the assumption that $\Lambda$ is confined. The charmenability of   $\Gamma$ implies that 
$X$ carries an ergodic $\Gamma$-invariant Borel probability measure $\nu$  \cite[Proposition 3.5]{BBHP}. Theorem 
\ref{thm:confined ergodic action of lattice} implies that $\nu$-almost every subgroup   has finite index in $\Gamma$. This certainly implies that \emph{some} subgroup $\Delta \in \overline{\Lambda^\Gamma}$ has finite index in $\Gamma$. Since finite-index subgroups are isolated points of $\Sub{\Gamma}$, we conclude that the subgroup $\Lambda$ itself has finite index in $\Gamma$ to begin with.

The second statement of the theorem follows just as in the previous paragraph.
\end{proof} 
 
 \begin{remark}
Consider the special case where  $G$ is a standard \emph{simple} group. Then $G$ has property (T), hence the lattice $\Gamma$ is \emph{charfinite} by \cite[Theorem A]{BBH}. Every uniformly recurrent subgroup of a charfinite group with trivial amenable radical is finite \cite[Proposition 3.5]{BBHP}. In that case, we may invoke the classical Margulis normal subgroup theorem to conclude that $X$ is the set of conjugates of some finite-index subgroup.
\end{remark}

\begin{proof}[Proof of Theorem \ref{thm intro: lattices} and   Corollary \ref{cor:lattices URS statement }]
These two statements from the introduction  follow as special cases of  Theorem \ref{thm: lattices, general}, by noting that connected center-free semisimple Lie groups  are   standard semisimple groups over $\mathbb{R}$ \cite[3.1.6]{zimmer2013ergodic}.
\end{proof}





\section{Margulis functions on factors}
\label{sec:margulis functions}

The current section deals with the notion of Margulis functions \cite{eskin1998upper,margulisfcns} (also called Foster--Lyapunov functions, see e.g. \cite{meyn2012markov}). We first discuss this idea in the abstract. We then apply these functions to study sequences of vectors which are asymptotically invariant with respect to a single factor of a semisimple real Lie group, relying on methods from the work \cite{gelander2022effective}. This is  instrumental towards dealing with general strongly confined subgroups of products in \S\ref{sec:strongly confined}.

\subsection*{Abstract properties of Margulis functions}

\emph{All throughout this first part of \S\ref{sec:margulis functions}, 
let $G$ be a second countable locally compact group admitting a continuous action on a locally compact $\sigma$-compact topological space $X$.} Let $\mu$ be a compactly supported symmetric probability measure on the group $G$ whose support generates $G$.  

\begin{defn}
\label{defn:Margulis definition}
Let $\Phi : X \to \left[a,\infty\right)$ be a proper continuous map for some $ a > 0$. The map $\Phi$ is a \emph{$(\mu,c,b)$-Margulis function} for some $0 < c < 1$ and $b > 0$ if 
\begin{equation}
\label{eq:Margulis functio}
\mu * \Phi(x) = \int_G \Phi(gx) \; \mathrm{d} \mu(g) < c \Phi(x) + b
\end{equation}
for every point $x \in X$.
\end{defn}

The above Definition \ref{defn:Margulis definition} coincides with   the
authoritative \cite[Definition 1.1]{margulisfcns}, with an  additional assumption that  the Margulis function $\Phi$ is not allowed to  take the value $\infty$.
We will sometimes drop the  constants $c$ and $b$ from our notations and refer to $\Phi$ as a $\mu$-Margulis function, or simply as a Margulis function.

It is useful to note that by reiterating the convolution operator, it is possible to   \enquote{improve} a given Margulis function.

\begin{lemma}
 \label{lemma:improving Margulis function constant}
Let $\Phi$ be a $(\mu,c,b)$-Margulis function. There is a sequence 
$ 0 < L_1 < L_2 < L_3 < \cdots$ such that for each $i \in \mathbb{N}$  we have
\begin{equation}
\label{eq:L_i}
\mu^{*i} * \Phi(z) < \left(\frac{1+c}{2}\right)^i \Phi(z)
\end{equation}
for all points $z \in X$ satisfying $\Phi(z) \ge L_i$.
\end{lemma}
\begin{proof}
We will construct the sequence $L_i$ inductively. Taking $L_1 = \frac{2b}{1-c}$ handles the base of the induction, as follows directly from Equation \ref{eq:Margulis functio}. Next, arguing by induction, assume that the constant $L_i$ for some $i \in \mathbb{N}$ has been defined in such a way that Equation \ref{eq:L_i} is satisfied. Consider the closed subset $X_i = \Phi^{-1}(\left[a,L_i\right])$. The subset $X_i$ is compact as the Margulis function $\Phi$ is proper. Take $L_{i+1} > L_i$ to be any sufficiently large constant so that $\Phi(\mathrm{supp}(\mu) X_i) \subset \left[a,L_{i+1}\right)$. This implies that any point $z \in X$ with $\Phi(z) \ge L_{i+1}$ also satisfies $\mu * \Phi(z) \ge L_i$. The validity of Equation \ref{eq:L_i} with respect to $i+1$ readily follows from the induction assumption.
\end{proof}

We now  consider   the situation where the space $X$ is equipped with a $\mu$-stationary probability measure.

\begin{lemma}
\label{lemma:making a Margulis function L2}
Let $\Phi$ be a $(\mu,c,b)$-Margulis function. Set $B = \frac{b}{1-c}$. Assume that $X$ admits a $\mu$-stationary probability measure $\nu$. 
\begin{enumerate}
    \item We have
$$ \nu(\{x \in X \: : \: \Phi(x) \ge M \}) < \frac{B}{M}$$
for all $M > 0$.
\item The Margulis function $\Phi$ satisfies $\Phi^{\frac{1}{2}} \in L^1(X,\nu)$.
\end{enumerate}
\end{lemma}
\begin{proof}
(1) is \cite[Lemma 2.1]{gelander2022effective}.  
We will now present\footnote{The  proof of part (2) in Lemma \ref{lemma:making a Margulis function L2} is  taken from the proof of \cite[Proposition 9.2]{gelander2022effective}. We have chosen to reproduce it here, for the convenience of the reader and for self-containedness. The difference in notations makes it difficult to quote that proposition from \cite {gelander2022effective}  as-is.} a proof of (2). 

Let $D : \left[0,1\right] \to \mathbb{R}_{\ge 0}$ be the inverse   cumulative distribution function corresponding to $\Phi^{\frac{1}{2}}$. It is defined as
$$
D(t) = \inf \{ s \in \mathbb{R}_{\ge 0} \: : \: \nu(\{ x \in X   \: : \: \Phi^{\frac{1}{2}}(x) \le s\}) \ge t \}. $$
In other words, the function $D$ satisfies
$$ \nu( \{ x \in X \: : \: \Phi^{\frac{1}{2}}(x) \le  D(t) \} ) = t \quad \forall t \in \left[0,1\right].$$
Substituting $s=M^{-1}$ in Part (1) of the current lemma, we get that for every $s>0$
$$
 \nu (\{x \in X \: : \: \Phi^{\frac{1}{2}}(x)  \le s^{-\frac{1}{2}}\}) \ge  1-   B s
$$
with the above  constant $B$. 
It follows that
$ D(1 - B s) \le s^{-\frac{1}{2}} $
for all $s\in [0,B^{-1}]$. In other words
$$ D(t)\le B^{\frac{1}{2}}(1-t)^{-\frac{1}{2}}$$ for all $t\in [0,1]$.
It follows from Fubini's theorem applied to the ``area under the graph" of the function $\Phi^{\frac{1}{2}}$ regarded as a subset of $X \times \mathbb{R}_{\ge 0}$ that
\begin{equation}
\label{eq:Fubini}
 \int_X \Phi^{\frac{1}{2}}(x) \; \mathrm{d} \nu(x) = \int_{\left[0,1\right]}  D(t) \; \mathrm{d} \lambda(t)
 \end{equation}
where $\lambda$ is the Lebesgue measure.  As the  integral $B^\frac{1}{2}\int_0^1 (1-t)^{-\frac{1}{2}} \mathrm{d}t$ converges, 
we conclude from Equation (\ref{eq:Fubini})   that the function $\Phi^{\frac{1}{2}}$ is indeed $\nu$-integrable.
 \end{proof}


The concave variant of Jensen's inequality implies that if $\Phi$ is a $(\mu,c,b)$-Margulis function then $\Phi^\alpha$ is a $(\mu,c^\alpha,b^\alpha)$-Margulis function for every exponent $0 < \alpha < 1$.
This observation coupled Lemma \ref{lemma:making a Margulis function L2} allows us to assume without loss of generality that we are working with an $L^1$  (or $L^2$) Margulis function to begin with.

%


We shall now consider the more restrictive situation, where the space $X$ admits an invariant probability measure.

\begin{lemma}
\label{lemma:controlling operator norm of Margulis functions}
Let $\nu$ be a $G$-invariant probability measure on $X$. Let $\Phi$ be a $(\mu,c,b)$-Margulis function with $\Phi \in  L^2(X,\nu)$.
Let $ 0 < L_1 < L_2 < \cdots $ be the constants provided by  Lemma \ref{lemma:improving Margulis function constant} and denote $\Omega_i = \Phi^{-1}(\left[L_i,\infty\right))$ for each $i \in \mathbb{N}$.
Let $\mathrm{P}_i $ be the orthogonal projection operator in $L^2(X,\nu)$ given by
$$ \mathrm{P}_i : f \mapsto f \cdot 1_{\Omega_i} \quad \forall f \in L^2(X,\nu)$$
Then $\|\mathrm{P}_i  \mu \mathrm{P}_i \|_{\mathrm{op}} \le \left(\frac{1+c}{2}\right)^i$ for each $i \in \mathbb{N}$.
\end{lemma}
This is a reminiscent of \cite[\S9]{gelander2022effective}, see   the proof of Theorem 9.3 there.

\begin{proof}[Proof of Lemma \ref{lemma:controlling operator norm of Margulis functions}]
Let an arbitrary $i \in \mathbb{N}$ be fixed. To ease our notations, set $\mathrm{P} = \mathrm{P}_i$ and $C = \left(\frac{1+c}{2}\right)^i$. Our goal will be to show that $\|\mathrm{P} \mu \mathrm{P}\|_{\mathrm{op}} \le C$. The spectrum of the operator $\mathrm{P}  \mu \mathrm{P} $  is a closed subset of the real interval $\left[-1,1\right]$.
Assume towards contradiction that this spectrum   admits a point $\lambda \in \left[-1,1\right]$ with $|\lambda| > C > 0$.  We can find a sequence of functions $f_n \in L^2(X,\nu)$ with $\|f_n\| = 1$  and a sequence of real numbers   $\varepsilon_n > 0$ with $ \varepsilon_n \to 0$   such that $$\|\mathrm{P}\mu\mathrm{P}  f_n - \lambda f_n \| < \varepsilon_n. $$
The fact that $\Phi$ is a $(\mu, c, b)$-Margulis function means that the inequality
$$ \mathrm{P} \mu \mathrm{P} \Phi \leq C \Phi$$
holds true $\nu$-almost everywhere, see Lemma \ref{lemma:improving Margulis function constant}. Since the convolution operator $\mu$ is symmetric, the composition $\mathrm{P} \mu \mathrm{P}$ is a self-adjoint operator. For each $n$  
\begin{align*}
C \left< \Phi, |f_n| \right> &\ge \left< \mathrm{P} \mu \mathrm{P} \Phi, |f_n| \right> = \left< \Phi,  \mathrm{P} \mu \mathrm{P} |f_n| \right> \ge \\
&\ge \left< \Phi, | \mathrm{P} \mu \mathrm{P} f_n| \right> = \left< \Phi, |\lambda f_n +( \mathrm{P} \mu \mathrm{P} - \lambda) f_n| \right> \ge \\
&\ge |\lambda|  \left< \Phi, |f_n| \right> -  \left< \Phi, |( \mathrm{P} \mu \mathrm{P} - \lambda) f_n| \right> \ge \\
&\ge  |\lambda|  \left< \Phi, |f_n| \right> - \|\Phi\| \cdot \| (\mathrm{P} \mu \mathrm{P} - \lambda) f_n\| \ge |\lambda| \left< \Phi, |f_n| \right> - \varepsilon_n \|\Phi\|.
\end{align*}
Note that $\left<\Phi, |f_n|\right> > 0$ for otherwise $\mathrm{P} f_n = 0$. We arrive at a contradiction to the assumption that $|\lambda| > C$.
\end{proof}

\begin{lemma}
\label{lemma:quantitaive no small elements}
Maintain all the assumptions and notations of Lemma \ref{lemma:controlling operator norm of Margulis functions}. In addition, let $f_n \in L^2(X,\nu)$ be an asymptotically $G$-invariant sequence of unit vectors. Consider the sequence of probability measures $ m_n = |f_n|^2 \cdot \nu$ on the space $X$.
 Then 
 $$ \limsup_n m_n(\Omega_{i+1})  \le 4 \left(\frac{1+c}{2}\right)^{2i}$$
 holds true for each $i \in \mathbb{N}$.
\end{lemma}
\begin{proof}
Let an arbitrary $i \in \mathbb{N}$ be fixed. To ease notation set $C = \left(\frac{1+c}{2}\right)^i$. Let $\mathrm{P}$ and $\mathrm{P}'$   denote the two orthogonal projections in $L^2(X,\nu)$ given  by $f \mapsto f\cdot 1_{\Omega_i}$ and $f \mapsto f \cdot 1_{\Omega_{i+1}}$ respectively. Recall that $\mathrm{supp}(\mu) \Omega_{i+1} \subset \Omega_i$. Hence  $\mathrm{P}   \mu  \mathrm{P}'   = \mu  \mathrm{P}'  $ as well as $\mathrm{P}'\mathrm{P} = \mathrm{P} \mathrm{P}' = \mathrm{P}'$. 
On the one hand, the asymptotic $G$-invariance of the sequence $f_n$ means that $\|\mu  * f_n\| \to 1$ as $n\to\infty$. On the other hand, Lemma \ref{lemma:controlling operator norm of Margulis functions} implies
\begin{align*}
\| \mu  * f_n\| &= \|\mu  * (\mathrm{P}'  f_n + (1-\mathrm{P}') f_n) \| \le \\
&\le \| (\mathrm{P}  \mu  \mathrm{P} ) \mathrm{P}' f_n\| + \|\mu (1-\mathrm{P}') f_n\| \le 
 C \|\mathrm{P}' f_n\| + \|(1-\mathrm{P}') f_n\|.
\end{align*}
Let $x_n = \|\mathrm{P}' f_n \|$ so that  $\sqrt{1-x_n^2} = \|(1-\mathrm{P}') f_n \| $. We get
$$ \liminf_n (C  x_n + \sqrt{1-x_n^2}) \ge 1.$$
Note that $x_n^2 = \|\mathrm{P}'f_n\|^2 = m_n(\Omega_{i+1})$ for all $n$. Solving the above inequality for $x_n^2$ gives
$$ \limsup x_n^2 \le \left(\frac{2C }{C^{2 }+1}\right)^2 \le 4C^2$$
as required.
\end{proof}

\begin{cor}
\label{cor:bounding support of norm with depth}
In the situation of  Lemma \ref{lemma:quantitaive no small elements},  any accumulation point $m$ of the sequence of probability measures $m_n$ satisfies $m(X)  = 1$.
\end{cor}

In other words, the probability measures $m_n$ do not  have \enquote{escape of mass at  infinity}.

\subsection*{A Margulis function on  discrete subsets with the Zassenhaus property}
In \cite[Theorem 1.5]{gelander2022effective} it is shown that a certain function depending on the injectivity radius is a Margulis function on the space of discrete subgroups of a standard semisimple group.
That function is denoted $\mathcal{I}^{-\delta}$ and the  result is termed \emph{the key inequality}.
An inspection of that proof  shows that the domain of definition of the Margulis function $\mathcal{I}^{-\delta}$ can be slightly extended. 
In this subsection we recall the setting of \cite{gelander2022effective} and present this extension. 

In the  discussion of the key inequality in \cite{gelander2022effective} the authors regard an algebraic group over an arbitrary local field (of good characteristic). 
However, in the present work we are only interested in   characteristic zero. Moreover,   since the function $\mathcal{I}^{-\delta}$ is actually constant when working over a zero characteristic  non-Archimedean local field, we focus here only on the Archimedean case.
By  restricting  scalars we may assume that we are working over the reals.

Let $G$ be a standard semisimple group of type $\infty$.
Let $K$ be a maximal compact subgroup of $G$ and endow the Lie algebra $\mathfrak{g}=\mathrm{Lie}(G)$ with an $\mathrm{Ad}(K)$-invariant norm. Denote by $m_{K}$ the normalized Haar probability measure on the group $K$.
Fix a semisimple group element $s=s(G)\in G$ as defined in  \cite[Equation (6.22)]{gelander2022effective}.
 Let $$\mu_s = \frac{1}{2} m_{K} * ( \delta_s + \delta_{s^{-1}})
* m_{K}$$ be the corresponding  probability measure on the group $G$. Note that $\mu_s$ is  symmetric\footnote{In some sections of  \cite{gelander2022effective} the authors work with the non-symmetric measure $m_K * \delta_s * m_K$. However, as $\mu_s = \frac{1}{2}(m_K * \delta_s * m_K + m_K * \delta_{s^{-1}} * m_K$), the results easily carry over.} and compactly supported.

Next, we will fix three positive real parameters ($\delta$, $R$ and $\rho$) and two identity neighborhoods ($V$ and $V_0$) associated with the group $G$.
Let $\delta=\delta(G)$ be the parameter given in \cite[Equation (6.20)]{gelander2022effective}.
Let $R=R(G)$ be the radius given by the Zassenhaus lemma.

\begin{lemma}[Zassenhaus lemma {\cite{zassenhaus1937beweis} or \cite[Lemma 2]{kavzdan1968proof}}] \label{lem:Zass}
There exists a radius $R=R(G)>0$ such that for every discrete subgroup $\Gamma<G$,
the subset $$\{\gamma\in \Gamma \: : \:   \text{$ \gamma=\exp(X)$ for some $ X\in \mathfrak{g}$ with $\|X\|\leq R$} \}$$ is contained in some connected nilpotent Lie subgroup of $G$.
\end{lemma}
\noindent We consider the identity neighborhood $$V=\exp\{X \in\mathfrak{g} \: : \: \|X\|\leq R\}$$ in the group $G$. 
It is called the \emph{Zassenhauss neighborhood}.
Lastly, fix a sufficiently small radius $0<\rho=\rho(G)<R$ such that the identity neighborhood $$V_0=\exp\{X \in\mathfrak{g} \: : \:   \|X\|\leq \rho\}$$ satisfies $$V_0\subset V\cap V^s\cap V^{s^{-1}}.$$ 
For all this, we refer to \cite[Equations (7.3), (7.4) and (7.5)]{gelander2022effective}.

Consider the space of discrete subgroups of the Lie group $G$. This space will be denoted $\sub_d(G)$ and regarded as an open subset of the Chabauty space $\Sub{G}$. 
On the space $\sub_d(G)$ we define the function $\mathcal{I}:\sub_d(G) \to (0,\rho]$  by
\[ \mathcal{I}(\Gamma)=\sup\{0< r \leq \rho \: : \: \Gamma\cap \exp \{X\in\mathfrak{g} \: : \: \|X\|\leq r\}=\{e\} \} \quad \forall \Gamma \in \sub_d(G).\]

\begin{theorem}
$\mathcal{I}^{-\delta}:\sub_d(G) \to [\rho^{-\delta},\infty)$
is a $\mu_s$-Margulis function.    
\end{theorem}

\begin{proof}
    We will see in Theorem~\ref{thm:key} below that the function $\mathcal{I}^{-\delta}$ is continuous and proper (in fact, it enjoys these two properties  over a larger domain of definition to be considered below). The function $\mathcal{I}^{-\delta}$ satisfies the   inequality in Equation (\ref{eq:Margulis functio}) with respect to the probability measure $m_K * \delta_s * m_K$ by \cite[Theorem 1.5]{gelander2022effective}. The same is true with respect to the probability measure  $m_K * \delta_{s^{-1}} * m_K$. Therefore $\mathcal{I}^{-\delta}$ is a $\mu_s$-Margulis function, as the symmetric probability measure $\mu_s$ is a convex combination of these two probability measures.
\end{proof}

We are now ready to extend the domain of definition of the Margulis function $\mathcal{I}^{-\delta}$. Consider the space $\cl(G)$ consisting of all closed subsets of the group $G$. The space  $\cl(G)$ is endowed   with the Fell topology \cite[\S12.C]{kechris2012classical}. The Chabauty space   $\sub(G)$ is a closed subspace of $\cl(G)$.

\begin{defn}
A closed subset $A\in \cl(G)$ has the \emph{Zassenhaus property} if 
    for every element $g\in G$, the subset $  A^g\cap V  $ is contained in some connected nilpotent Lie subgroup of $G$
    and  satisfies 
    \begin{equation}
        \label{eq:Za}
            (A^g\cap V)^2 \subset A^g
    \end{equation}
where $V \subset G$ is the Zassenhauss neighborhood introduced above.
\end{defn}

We denote by $\cl^\mathrm{Z}(G)$ the subset of $\cl(G)$ consisting of all sets with the Zassenhaus property.
This is   a closed and $G$-invariant subset of $\cl(G)$.
We denote by $\cl^\mathrm{Z}_d(G)$ the 
subset of $\cl^\mathrm{Z}(G)$ consisting of sets  containing the identity element of the group $G$ as an isolated point.
This is   an open and $G$-invariant subset of $\cl^\mathrm{Z}(G)$, which contains $\sub_d(G)$ by the Zassenhaus lemma (Lemma~\ref{lem:Zass}).
The function $\mathcal{I}$ extends naturally to the space $\cl^\mathrm{Z}_d(G)$. We set a function  
$\mathcal{J}:\cl^\mathrm{Z}_d(G) \to (0,\rho]$ to be given by
\[ \mathcal{J}(A)=\sup\{0< r \leq \rho \: : \: A\cap \exp \{X\in\mathfrak{g} \: : \: \|X\|\leq r\}=\{e\} \} \quad \forall A \in \cl^\mathrm{Z}_d(G).\]

\begin{theorem} \label{thm:key}
$\mathcal{J}^{-\delta}:\cl^\mathrm{Z}_d(G) \to [\rho^{-\delta},\infty)$
is a $\mu_s$-Margulis function.    
\end{theorem}

\begin{proof}
The continuity of the function $\mathcal{J}$ follows from the definition of the Chabauty topology, from the property given in Equation~\eqref{eq:Za}  and from the fact that the Lie group $G$ has no small subgroups.
The fact that the function  $\mathcal{J}^{-\delta}$ is proper is equivalent to saying that $\mathcal{J}(A_n) \to 0$ for every sequence of subsets $A_n \in \cl^\mathrm{Z}_d(G)$ converging to a subset $A \in \cl^\mathrm{Z}(G)$ in which the identity element is not an isolated point. This later statement clearly holds true.

We are left to show that the function $\mathcal{J}^{-\delta}$ satisfies Equation~\eqref{eq:Margulis functio}.
Observe that Equation (7.8) as well as Propositions 7.3, 7.4 and 7.5 in \cite{gelander2022effective}, originally formulated for the function $\mathcal{I}$, all apply mutatis mutandis to our new function $\mathcal{J}$. Indeed, the arguments of \cite{gelander2022effective} involve studying  finite collections of elements all belonging to some connected nilpotent Lie subgroup, regardless on whether these elements come from any particular enveloping discrete subgroup.
Thus, the desired result follows by  the same proof as that of the key inequality in \cite[\S8]{gelander2022effective}.
    \end{proof}

\subsection*{Margulis functions   on discrete subgroups of  products}

We let $G_1$ and $G_2$ be a pair of standard semisimple Lie groups of type $\infty$
and set $G=G_1\times G_2$. Let $\mathfrak{g_i} = \mathrm{Lie}(G_i)$ be the corresponding semisimple Lie algebras.
Consider the element $s_1=s(G_1)\in G_1$, the probability measure $\mu_{s_1}\in \Prob(G_1)$ and the positive real parameters $\delta_1=\delta(G_1)$, $R_1=R(G_1)$ and $\rho_1=\rho(G_1)$ as defined in the previous subsection.
In addition, consider the radius $R_2=R(G_2)$ and the corresponding identity neighborhoods 
\[ V_i=\exp\{X \in\mathfrak{g}_i \: : \: \|X\|\leq R_i\} \subset G_i \]
for $i \in \{1,2\}$ and $V=V_1\times V_2\subset G$. 
Consider the   map
\[ \Phi : \cl(G) \to \cl(G_1 ), \quad \Phi : A\mapsto \mathrm{pr}_1(A\cap (G_1\times V_2)) \quad \forall A \in \Cl{G}.\]
The map $\Phi$ is $G_1$-equivariant and Borel.

We consider the $G$-invariant Chabauty open subset $\sub_d(G)$ of $\sub(G)$ consisting of all discrete subgroups, and its   $G_1$-invariant open subset
\[ \mathcal{X}_1(G) =\{ \Gamma\in \sub_d(G) \: : \: \Gamma\cap (\{e_1\}\times V_2)=\{(e_1,e_2)\}\}.\]
As $V$ is an identity neighborhood of the radius provided by  Lemma~\ref{lem:Zass} (i.e. $V$ is a  Zassenhaus neighborhood), we get  that $\Phi(\mathcal{X}_1(G))\subset \cl^\mathrm{Z}_d(G_1)$. We can further apply the function $\mathcal{J}$ to this image.
We obtain the function 
$$\mathcal{L} : \mathcal{X}_1(G) \to \left(0,\rho_1\right], \quad \mathcal{L} =\mathcal{J}\circ \Phi.$$
While the map $\Phi$ is not continuous, it turns out that the composed map $\mathcal{L}$ is.
Moreover, we get the following.

\begin{theorem} \label{thm:key2}
$\mathcal{L}^{-\delta_1}:\mathcal{X}_1(G) \to [\rho_1^{-\delta_1},\infty)$
is a $\mu_{s_1}$-Margulis function.    
\end{theorem}

\begin{proof}
The fact that the map $\mathcal{L}^{-\delta_1}$ is continuous and proper follows is the same way as the proof of the same facts for the map $\mathcal{J}^{-\delta_1}$ in Theorem \ref{thm:key}. The integral inequality in Equation~\eqref{eq:Margulis functio} follows from Theorem~\ref{thm:key} combined with the $G_1$-equivariance of the map $\Phi$.
\end{proof}

\begin{cor}
\label{cor:discsub}
Let $\nu$ be a   
discrete  irreducible invariant random subgroup of the semisimple Lie group $G = G_1 \times G_2$. 
Let $f_n \in L^2(\Sub{G},\nu)$ be an asymptotically $G_1$-invariant sequence of \emph{unit} vectors.  Then any weak-$*$ accumulation point of the probability measures $|f_n|^2 \cdot \nu \in \mathrm{Prob}(\Sub{G})$ is almost surely discrete. 
\end{cor}
\begin{proof}
We have that $\nu(\mathcal{X}_1(G))=1$ by the assumption that the invariant random subgroup $\nu$ is irreducible. The desired conclusion follows immediately from Corollary \ref{cor:bounding support of norm with depth} applied to $\mathcal{L}^{-\delta_1}$,
which is a Margulis function by Theorem~\ref{thm:key2}.
\end{proof}

\section{Strongly confined subgroups of semisimple groups}
\label{sec:strongly confined}


Throughout this section, let $G$  be a connected semisimple Lie group of real rank at least two, without compact factors and with trivial center. In terms of the terminology introduced in \S\ref{sec:standard semisimple group}, this means that $G$ is a standard semisimple group of type $\infty$ and of rank at least two. By Zariski topology on $G$ we will refer to its real Zariski topology.
Our main goal will be to prove the following result.

\begin{theorem}
\label{thm:general with non-zarsiki dense intersection}
Let $\Lambda$ be a discrete subgroup of $G$. Then $\Lambda$ is an irreducible lattice in $G$ if and only if the following two conditions hold:
\begin{enumerate}
\item the subgroup $\Lambda$ is strongly confined, and
\item no  pair of    non-trivial  normal subgroups $H_1,H_2 \le G$ with $H_1 \cap H_2 = \{e\}$ satisfies 
$$ \overline{\Lambda \cap H_1}^\mathrm{Z} = H_1 \quad \text{and} \quad \overline{\Lambda \cap H_2}^\mathrm{Z} = H_2.$$
\end{enumerate}
\end{theorem}

As will be  explained below,   Theorem \ref{thm:general with non-zarsiki dense intersection} is stronger than and immediately implies Theorem \ref{thm intro:generalLie} stated in the introduction. Note that if the Lie group  $G$ is simple (rather semisimple) then it has Kazhdan's property (T), in which case Theorem \ref{thm:general with non-zarsiki dense intersection} follows from the main result of \cite{fraczyk2023infinite}. 





\subsection*{Local rigidity of irreducible lattices}

\begin{defn}
A given subgroup $\Lambda \le G$ is \emph{Chabauty locally rigid} if it admits a Chabauty neighborhood $\Lambda \in \Omega \subset \Sub{G}$ such that any subgroup $\Lambda \in \Omega$ is in fact conjugate to $\Lambda$ in $G$.
\end{defn}

This notion was introduced and studied in \cite{gelander2018invariant}, where  \enquote{classical} local rigidity \cite{selberg1960discontinuous,calabi1961compact,weil1962discrete} was leveraged to obtain the following.

\begin{theorem}[Theorem 1.10 of \cite{gelander2018invariant}]
\label{thm:Chabauty local rigidity}
Every irreducible lattice $\Gamma$ in the semisimple Lie group $G$ is Chabauty locally rigid.
\end{theorem}

\subsection*{Stationary limits and random subgroups}

Let $\mu = \mu_s$ be the symmetric compactly supported probability measure on the Lie group $G$ considered in \cite{gelander2022effective} and discussed explicitly in the above 
 \S\ref{sec:margulis functions}.
 

\begin{defn}
\label{def:mu-stationary limit}
A $\mu$-\emph{stationary limit} of a given discrete subgroup $\Lambda \le G$ is any weak-$*$ accumulation point $\nu$ of the sequence of Ces\'aro averages $\frac{1}{n} \sum_{1=1}^n \mu^{*n} * \delta_{\Lambda}$ in the space $\Prob(\Sub{G})$.
\end{defn}

It is important to note that if $\nu$ is any $\mu$-stationary limit of a given discrete subgroup $\Lambda$ then $\mathrm{supp}(\nu)\subset \overline{\Lambda^G}$. In particular $\nu$-almost every subgroup is a conjugate limit of $\Lambda$.

Any $\mu$-stationary limit of a discrete subgroup of the semisimple Lie group $G$ is almost surely discrete by \cite[Theorem 2.2]{fraczyk2023infinite} (see also \cite[Corollary 1.6]{gelander2022effective}).

The following deep stiffness result provides a detailed classification of discrete $\mu$-stationary random subgroups of semisimple Lie groups. It relies on the celebrated structure theory of Nevo and Zimmer \cite{nevo1999homogenous,nevo2002structure}.

\begin{theorem}[Theorem 6.5 of \cite{fraczyk2023infinite}]
\label{thm:classification of SRS in products}
Let $\nu$ be an ergodic discrete $\mu$-stationary random subgroup of the semisimple Lie group $G$. Then the  group $G$  decomposes as a product of three semisimple factors $G = G_\mathcal{I} \times G_\mathcal{H} \times G_\mathcal{T}$
such that
\begin{enumerate}
    \item $\nu$  projects to an invariant random subgroup in $G_\mathcal{I}$ for which all the irreducible factors are of rank at least 2,
    \item $G_\mathcal{H}$ is a product of rank one factors and $\nu$ projects discretely to every factor
of $G_\mathcal{H}$, and
\item $\nu$ projects trivially to $G_\mathcal{T}$.
\end{enumerate}
Furthermore, the intersection of $\nu$-almost every  subgroup with every simple factor
of $G_\mathcal{H}$ as well as with every irreducible factor of $G_\mathcal{I}$ is Zariski-dense in that factor.
\end{theorem}

The notion of an \emph{irreducible factor} of an invariant random subgroup is introduced in \cite[p. 401]{fraczyk2023infinite}. Namely, given an ergodic discrete invariant random subgroup  $\nu$ of $G$, there is a direct product decomposition $G = H_1 \times \cdots \times H_k$ such that $\nu$-almost every subgroup projects to  each $H_i$ discretely but  projects to each proper factor of $H_i$ densely, see \cite[Theorem 4.1]{fraczyk2023infinite}. These $H_i$'s are the irreducible factors corresponding to $\nu$.

\begin{remark}
Unfortunately, we are not aware of a local field version of Nevo and Zimmer work \cite{nevo1999homogenous,nevo2002structure} in the existing literature. While we do not expect this to be a significant obstacle, in the current state of affairs we consider only real Lie groups in \S\ref{sec:strongly confined}.
\end{remark}

We mention a deep result which plays a crucial role in our analysis. Recall that an invariant random subgroup is called \emph{irreducible} if every non-trivial normal subgroup is acting ergodically.

\begin{theorem}[Stuck--Zimmer \cite{SZ}, Hartman--Tamuz \cite{Hartman-Tamuz}]
\label{thm:hartman-tamuz}
Let $\nu$ be a non-trivial   irreducible  invariant random subgroup of the semisimple Lie group $G$. Then $\nu$-almost every subgroup is coamenable.
\end{theorem}

Here is a closely related statement.

\begin{cor} 
\label{cor:SZ-P}
Let $\nu$ be a non-trivial irreducible invariant random subgroup of the semisimple Lie group $G$. Assume that the $G$-space $(\Sub{G},\nu)$ has spectral gap. Then $\nu$-almost every subgroup is a lattice.
\end{cor}
\begin{proof}
Since the action of $G$ on $(\Sub{G},\nu)$ has spectral gap it cannot be properly ergodic. This fact is  a consequence of Theorem \ref{thm:hartman-tamuz}. We refer to  \cite[Proposition 7.6]{creutz2017stabilizers} or \cite[Theorem 3]{levit2020benjamini} for the full details concerning this implication. Therefore the action of $G$ on $(\Sub{G},\nu)$ is essentially transitive. As such, the stabilizer (i.e. the normalizer) 
 of $\nu$-almost every subgroup $\Lambda$ is  a lattice    \cite[Lemma 3.5]{SZ}. In other words $\nu$-almost every subgroup is a non-trivial normal subgroup of a lattice. We  conclude by the Margulis normal subgroup theorem (or by our Theorem \ref{thm intro: lattices}).
\end{proof}


We remark that the standing higher rank  assumption is crucial both in  Theorem \ref{thm:hartman-tamuz} and in its Corollary \ref{cor:SZ-P}. 

\subsection*{From strongly confined  to  coamenable}
Recall that $G$ is a connected, center-free semisimple Lie group without compact factors and of rank at least two. Let us recall the following notion  (introduced in Definition \ref{def:strongly and irreducibly confined - intro} of the introduction in a more general setting).

\begin{defn}
A subgroup $\Lambda$ of $G$ is \emph{irreducibly confined} if $\Lambda$ is strongly confined (in the sense of Definition \ref{defn:strongly confined}) and furthermore 
    the  intersection $\Lambda \cap H$ is trivial for   any non-trivial proper normal subgroup $H\lhd G$.
\end{defn}

\begin{remark}
The two notions of confined and strongly confined are clearly closed with respect to the Chabauty topology. We do not know if the notion of irreducibly confined is closed in general. It is closed at least in the case where the group $G$ is a direct product of rank one simple factors (as follows from Corollary \ref{cor:special cases of conjugate limit with Zariski dense intersection} combined with Lemma  \ref{lemma:discrete irr confined has Zariski dense projections}).
\end{remark}

\begin{thm}
\label{thm:analysis of stationary limits}
Let $\Delta \le G$ be a strongly  confined discrete subgroup. Assume that there is no  pair of    non-trivial  normal subgroups $H_1,H_2 \le G$ such that $H_1 \cap H_2 = \{e\}$, 
$ \overline{\Delta \cap H_1}^\mathrm{Z} = H_1$ and $\overline{\Delta \cap H_2}^\mathrm{Z} = H_2$. 
Then there is a discrete     irreducible invariant random subgroup $\nu$ of the group $G$ with $\mathrm{supp}(\nu) \subset \overline{\Delta^G}$. Further $\nu$-almost every subgroup is irreducibly confined.
\end{thm}
\begin{proof}
Let $\nu$ be any $\mu$-stationary  limit of the subgroup $\Delta$. We know that $\nu$-almost every subgroup is discrete by  \cite[Theorem 1.6]{fraczyk2023infinite}. Up to replacing $\nu$ by a generic ergodic component, we may assume that $\nu$ itself is ergodic.

 We will now use  the stiffness result (Theorem \ref{thm:classification of SRS in products}) and its notation to analyze the resulting discrete ergodic $\mu$-stationary  random subgroup $\nu$. The fact that the factor $G_\mathcal{T}$ is trivial follows as the subgroup $\Delta$ is strongly confined.

 We claim that $G = G_\mathcal{I}$ and that $G_\mathcal{I}$ does not have proper non-trivial irreducible factors. Indeed, in any other situation, either the factor   $G_\mathcal{H}$ will be non-trivial or the factor $G_\mathcal{I}$ will have more than a single irreducible factor. In both cases, the group $G$ itself must be semisimple but not simple, and it can we written as a non-trivial direct product $G = G_1 \times G_2$ in such a way that $\nu$-almost every discrete subgroup $\Lambda$ satisfies $\overline{\Lambda \cap G_1}^\mathrm{Z} = G_1 $ as well as 
 $\overline{\Lambda \cap G_2}^\mathrm{Z} = G_2 $. Since $\nu$-almost every subgroup is a conjugate limit of the subgroup $\Delta$, this would lead to a contradiction to part (2) of Corollary \ref{cor:special cases of conjugate limit with Zariski dense intersection}.

To conclude, it follows from the stiffness result (Theorem \ref{thm:classification of SRS in products}) that $\nu$ is an irreducible invariant random subgroup. The fact that $\nu$-almost every subgroup is strongly confined  follows because $\nu$-almost every subgroup is a conjugate limit of $\Delta$. Lastly $\nu$-almost every subgroup intersects trivially every proper semisimple factor by the irreducibility of $\nu$, and as such is irreducibly confined.
\end{proof}
 
 The above result has the following statement   as an immediate special case.

 \begin{cor}
\label{cor:analysis of stationary limits - irreducibly confined}
Let $\Delta \le G$ be an irreducibly  confined discrete subgroup.  
Then there is a non-trivial discrete     irreducible invariant random subgroup $\nu$ of the group $G$ with $\mathrm{supp}(\nu) \subset \overline{\Delta^G}$. 
\end{cor}

Further, we obtain the following  result, which allows us to go from irreducibly confined subgroups to coamenable ones.

\begin{cor}
\label{cor:confined has a coameanble conjugate limit}
Let $\Delta \le G$ be an irreducibly   confined discrete subgroup. Then  $\Delta$ admits a coamenable conjugate limit in $\Sub{G}$.
\end{cor}
\begin{proof}
According to Corollary \ref{cor:analysis of stationary limits - irreducibly confined} the group $\Delta$ admits a non-trivial discrete   irreducible invariant random subgroup $\nu$ with $\mathrm{supp}(\nu) \subset \overline{\Delta^G}$. We know that $\nu$-almost every subgroup is coamenable in $G$ by Theorem \ref{thm:hartman-tamuz}. This concludes the proof.
\end{proof}


\subsection*{From a coamenable subgroup to a lattice}

As before, recall that $G$ is a connected, center-free higher rank semisimple Lie group without compact factors.
The following result can be regarded as a variant of the Stuck--Zimmer theorem for higher rank semisimple Lie groups \cite{SZ} without assuming Kazhdan's property (T) but with the added assumption of \enquote{strongly confined}.

\begin{theorem}
\label{theorem:strictly confined coamenable is a lattice}
Let $\nu$ be a discrete irreducible invariant random subgroup of the semisimple Lie group $G$. Assume that $\nu$-almost every subgroup is strongly confined. Then $\nu$-almost every subgroup  is a lattice in $G$.    
\end{theorem}

\begin{proof}
In view of Corollary \ref{cor:SZ-P} it is enough to show that 
the $G$-action on $(\Sub{G},\nu)$
has a spectral gap, namely, the unitary representation $L^2_0(\Sub{G},\nu)$ does not  almost have invariant vectors. 
We proceed to show that.
If the group $G$ has Kazhdan's property (T) then this is immediate. Thus we assume as we may that it does not. As such, the group $G$ can be written as a direct product $G = G_1 \times G_2$ where $G_1$ is some (non-trivial) semisimple Lie group and $G_2$ is a simple Lie group of real rank one.
We will conclude by using our spectral gap Theorem \ref{theorem:getting spectral gap - analytic groups} applied with respect to  the space $(\Sub{G},\nu)$.
It remains to verify its conditions.
\begin{itemize}
\item The fact that $L^2_0(\Sub{G},\nu)^{G_2} = 0$ follows from the irreducibility of $\nu$.
\item 
The stabilizer in $G$ of any subgroup $\Lambda\in\Sub{G}$ is its normalizer 
$\mathrm{N}_{G}(\Lambda)$.
The irreducibility assumption means that the factor $G_2$ acts ergodically. Hence $\mathrm{N}_{G}(\Lambda)\cap G_1=\mathrm{N}_{G_1}(\Lambda)$ is $\nu$-almost surely constant. This constant subgroup is $G_1$-invariant, i.e. a normal subgroup of $G_1$.
Since $\Lambda$ is not contained in any proper factor of $G$, it follows that $\mathrm{N}_{G_1}(\Lambda)=
\{e\}$.
\end{itemize}

We are left to verify the assumption in the third  bullet  of  Theorem \ref{theorem:getting spectral gap - analytic groups}. Consider any asymptotically $G_1$-invariant sequence of unit vectors $f_n \in L_0^2(\Sub{G},\nu)$ and  let $\eta \in \Prob(\Sub G)$ be an accumulation point of the sequence of  probability measures $\mathrm{Stab}_*(|f_n|^2 \cdot \nu)$ (where the stabilizer map $\mathrm{Stab} : \Sub{G} \to \Sub{G}$ is just the normalizer map). We need to show that $\eta$-almost every subgroup is  discrete,  not contained in the factor $G_2$  and admits Zariski dense\footnote{The statement of Theorem \ref{theorem:getting spectral gap - analytic groups} requires these projections to $G_2$ to be Zariski-dense and  and not relatively compact. However, Zariski density implies not relatively compact  when working with real Lie groups.} projections to $G_2$. These three properties follow respectively from 
 Corollary~\ref{cor:discsub}, the assumption that $\nu$-almost every subgroup is strongly confined and Lemma~\ref{lemma:discrete irr confined has Zariski dense projections}. To be precise,  Corollary~\ref{cor:discsub} is to be applied with respect to the invariant random subgroup $\zeta = \mathrm{Stab}_* \nu$ and a corresponding asymptotically $G_1$-invariant sequence of unit vectors $g_n \in L^2(G,\zeta)$ such that $|g_n|^2 \cdot \zeta = \mathrm{Stab}_*(|f_n|^2 \cdot \nu)$.
\end{proof}

We are ready to prove the main result of \S\ref{sec:strongly confined}.

\begin{proof}[Proof of Theorem \ref{thm:general with non-zarsiki dense intersection}]
Every irreducible lattice is strongly  confined by Lemma \ref{lemma:confined subgroup of an irreducible lattice is strongly confined}. In addition, every irreducible lattice is irreducibly confined, for it intersects trivially all proper factors. This conclusion is stronger than that in statement of Theorem \ref{thm:general with non-zarsiki dense intersection}.

We  now show the converse direction, which is the more interesting one. Let $\Lambda$ be a strongly confined discrete subgroup of $G$ satisfying the conditions in the statement of the theorem, namely that there no pair of commuting non-trivial factors of $G$ both admitting  Zariski-dense intersections with $\Lambda$.
By Theorem \ref{thm:analysis of stationary limits} there exists a non-trivial discrete  irreducible invariant random subgroup $\nu$ of the group $G$ with $\mathrm{supp}(\nu) \subset \overline{\Lambda^G}$. In particular $\nu$-almost every subgroup is irreducibly confined. According to Theorem \ref{theorem:strictly confined coamenable is a lattice} the invariant random subgroup $\nu$ arises from some irreducible lattice $\Gamma$ in the group $G$ (i.e. $\nu$ gives full measure to subgroups conjugate to $\Gamma$). In particular, this irreducible lattice $\Gamma$ is a conjugate limit of $\Lambda$. By the Chabauty local rigidity of irreducible lattices (Theorem  \ref{thm:Chabauty local rigidity}) we conclude that the subgroup $\Lambda$ itself must be an irreducible lattice, as required.
\end{proof}

\begin{proof}[Proof of Theorem \ref{thm intro:generalLie} of the introduction]
This follows immediately from Theorem \ref{thm:general with non-zarsiki dense intersection}. Indeed, note that any irreducibly confined discrete subgroup has trivial intersections with all proper factors, and as such certainly satisfies the (weaker) conditions of Theorem  \ref{thm:general with non-zarsiki dense intersection}.
\end{proof}

\begin{proof}[Proof of Corollary \ref{cor:intro: weak version} of the introduction]
Let $(X,m)$ be a strongly irreducible\footnote{Recall that in the introduction we defined a discrete subgroup to be irreducible if it projects densely to each proper factor. We defined a (discrete) subgroup to be strongly irreducible if every discrete conjugate limit of it irreducible.}
 probability measure preserving action of the semisimple Lie group $G$. Up to passing to a generic ergodic component, we may assume without loss of generality that $(X,m)$ is ergodic. Strong irreducibility and ergodicity implies irreducibility by \cite[Corollary 7.3]{fraczyk2023infinite}, so that the action on $(X,m)$ is irreducible. As in the first paragraph of the proof of Theorem \ref{theorem:strictly confined coamenable is a lattice}, it will suffice to prove that the action $(X,m)$ has spectral gap. Further, we may assume that $G$ is written as $G = G_1 \times G_2$ where $G_1$ is a non-trivial semisimple Lie group and $G_2$ is a simple Lie group of real rank one. We will deduce spectral gap directly from Theorem \ref{theorem:getting spectral gap - general case}. Let us verify the assumptions of that theorem.
 \begin{itemize}
     \item The fact that $L^2_0(X,m)^{G_2} = 0$ follows from the irreducibility of the action.
     \item Consider some asymptomatically $G_1$-invariant sequence of unit vectors   $f_n \in L^2(X,m)$ and let $\mu \in \Prob(\Sub G)$  be an  accumulation point of the sequence of probability measures $\mathrm{Stab}_*(|f_n|^2 \cdot m) \in \mathrm{Prob}(\Sub{G})$. Note that $\mu$-almost every subgroup is discrete by Corollary \ref{cor:discsub}. As $\mu$-almost every subgroup $\Lambda$ is a discrete conjugate limit of some $m$-generic subgroup, it satisfies $\overline{G_1 \Lambda} = G$ by the strong irreducibility assumption.
 \end{itemize}
This verifies the assumptions of Theorem   \ref{theorem:getting spectral gap - general case} and thereby concludes the proof.
\end{proof}

\subsection*{Products of general locally compact groups}

In this final subsection, we deviate from the standing assumptions of \S\ref{sec:strongly confined}, and let $G = G_1 \times G_2$ be a direct product of two locally compact second countable  groups. Assume that $G_2$ has a compact abelianization. We prove Theorem \ref{thm:strong irreducible} of the introduction. It says  that under certain irreducibility conditions, a   coamenable discrete subgroup of $G$ must be a lattice.

\begin{proof}[Proof of Theorem 
\ref{thm:strong irreducible}]
Let $\Lambda \le G$ be a discrete coamenable subgroup, such that there are no $G_2$-invariant vectors in $L^2_0(G/\Lambda)$ and that    every  conjugate limit of $\Lambda$ has dense projections to the factor $G_2$. 
It follows from Theorem \ref{theorem:getting spectral gap - general case} that the unitary $G$-representation $L^2_0(G/\Lambda)$ has a spectral gap. Since $\Lambda$ is coamenable this must mean that $\Lambda$ is a lattice (i.e. this situation is only possible provided $L^2_0(G/\Lambda)$ is a proper subrepresentation of $L^2(G/\Lambda).$
\end{proof}


\bibliographystyle{alpha}
\bibliography{irreducible}

\end{document}